\documentclass[onefignum,onetabnum]{siamonline190516}

\usepackage{lipsum}
\usepackage{amsfonts}
\usepackage{graphicx}
\usepackage{epstopdf}
\usepackage{algorithmic}
\ifpdf
  \DeclareGraphicsExtensions{.eps,.pdf,.png,.jpg}
\else
  \DeclareGraphicsExtensions{.eps}
\fi

\usepackage{enumitem}
\setlist[enumerate]{leftmargin=.5in}
\setlist[itemize]{leftmargin=.5in}

\newsiamremark{remark}{Remark}
\newsiamremark{hypothesis}{Hypothesis}
\crefname{hypothesis}{Hypothesis}{Hypotheses}
\newsiamthm{claim}{Claim}

\headers{Microlocal Analysis of a Compton Tomography Problem}{J.
Webber and E.T. Quinto}

\title{Microlocal Analysis of a Compton Tomography Problem\thanks{Submitted to the editors 01/14/2020.
\funding{The first author was supported by the U.S. Department of Homeland Security, Science and Technology Directorate, Office of University Programs, under Grant Award 2013-ST-061-ED0001. The work of the second author was partially  supported by the
U.S. National Science Foundation under Grant DMS 1712207.}}}

\author{James Webber\thanks{Department of Electrical and Computer
Engineering, Tufts University, Medford, MA USA 
  (\email{James.Webber@tufts.edu}}).
\and Eric Todd Quinto\thanks{Department of Mathematics, Tufts
University, Medford, MA USA 
  (\email{Todd.Quinto@tufts.edu}).}}

\usepackage{amsopn}

\graphicspath{{ToricFigs/}}

\usepackage{subcaption}

\usepackage{tikz}
\usetikzlibrary{quotes,angles}
\usetikzlibrary{decorations.pathmorphing}
\tikzset{snake it/.style={decorate, decoration=snake}}

\newcommand{\ra}{r,\alpha}
\newcommand{\Tsty}{T^*(Y)}
\newcommand{\ta}{\theta_{\alpha}}

\newcommand{\smo}{\setminus \mathbf{0}}
\newcommand{\zero}{\mathbf{0}}
\newcommand{\txi}{\tilde{\xi}}

\newcommand{\partxf}[2]{\frac{\partial #2}{\partial x_{#1}}}

\newcommand{\partxif}[2]{\frac{\partial #2}{\partial \xi_{#1}}}

\newcommand{\abs}[1]{\left|#1\right|}                 
\newcommand{\paren}[1]{\left(#1\right)}               
\newcommand{\bparen}[1]{\left[#1\right]}               
\newcommand{\sparen}[1]{\left\{#1\right\}}      

\newcommand{\dd}{\mathrm{d}}  

\newcommand{\supp}{\operatorname{supp}} 
\newcommand{\vp}{\varphi}

\newcommand{\Cc}{\mathcal{C}}
\newcommand{\tCc}{\widetilde{\mathcal{C}}}
\newcommand{\Dc}{\mathcal{D}}
\newcommand{\Ec}{\mathcal{E}}
\newcommand{\Fc}{\mathcal{F}}

\newcommand{\Tc}{\mathcal{T}}

\newcommand{\WF}{\mathrm{WF}}                         

\newcommand{\vb}{\mathbf{b}}
\newcommand{\vc}{\mathbf{c}}
\newcommand{\vs}{\mathbf{s}}
\newcommand{\vd}{\mathbf{d}}

\newcommand{\vv}{\mathbf{v}}
\newcommand{\vw}{\mathbf{w}}
\newcommand{\vx}{\mathbf{x}}
\newcommand{\vy}{\mathbf{y}}

\newcommand{\vg}{\mathbf{g}}

\newcommand{\otp}{[0,2\pi]}
\newcommand{\rr}{{{\mathbb R}}}
\newcommand{\zz}{{{\mathbb Z}}}

\newcommand{\rtwo}{{{\mathbb R}^2}}

\newcommand{\rn}{{{\mathbb R}^n}}

\newcommand{\st}{\hskip 0.3mm : \hskip 0.3mm}

\newcommand{\be}{\begin{equation}}
\newcommand{\ee}{\end{equation}}
\newcommand{\bea}{\begin{eqnarray}}
\newcommand{\eea}{\end{eqnarray}}
\newcommand{\bean}{\begin{eqnarray*}}
\newcommand{\eean}{\end{eqnarray*}}

\newcommand{\bel}[1]{\begin{equation}\label{#1}}

\newcommand{\eel}[1]{{\label{#1}\end{equation}}}

\DeclareMathOperator*{\argmin}{arg\,min}

\ifpdf
\hypersetup{
  pdftitle={Microlocal Analysis of a Compton Tomography Problem}
  pdfauthor={J. Webber and E.T. Quinto}
}
\fi

\begin{document}

\maketitle

\begin{abstract}
 Here we present a novel microlocal analysis of a new toric section
transform which describes a two dimensional image reconstruction
problem in Compton scattering tomography and airport baggage
screening. By an analysis of two separate limited data problems for
the circle transform and using microlocal analysis, we show that the
canonical relation of the toric section transform is 2--1. This
implies that there are image artefacts in the filtered backprojection
reconstruction. We provide explicit expressions for the expected
artefacts and demonstrate these by simulations. In addition, we prove
injectivity of the forward operator for $L^\infty$ functions supported
inside the open unit ball.  We present reconstructions from simulated
data using a discrete approach and several regularizers with varying
levels of added pseudo-random noise.
\end{abstract}



\section{Introduction} We consider the Compton scattering tomography
acquisition geometry displayed in figure \ref{fig1}, which illustrates
an idealized source--detector geometry in airport baggage screening
representing the Real Time Tomography (RTT) geometry \cite{will}.
See appendix \ref{app1} for more detail on the
potential for the application of this work in airport baggage
screening. The inner circle (of smaller radius) represents a ring of
fixed energy sensitive detectors and the outer circle a ring of fixed,
switched X-ray sources, which we will assume for the purposes of this
paper can be simulated to be monochromatic (e.g. by varying the X-ray
tube voltage and taking finite differences in energy or by source
filtering \cite{mono1,mono2}). It is noted that the RTT geometry is
three dimensional \cite{will}, but we assume a two dimensional
scattering geometry as done in \cite{me}. Further we note that in the desired application in airport baggage screening we expect the data to be very noisy. Later in section \ref{res} we simulate the noisy data using an additive Gaussian model with a significant level (up to $5\%$) and show that we can combat the noise effectively using the methods of \cite{IRtools} (specifically the ``IRhtv" method).

Compton scattering describes the inelastic scattering process of a photon
with charged particles (usually electrons). The energy loss is given by the
equation
\begin{equation}
\label{equ1}
E'=\frac{E}{1+(E/E_0)(1-\cos\omega)},
\end{equation}
\begin{figure}[!h]
\centering
\begin{tikzpicture}[scale=4]
\coordinate (S) at ({-cos(30)},0.5);
\coordinate (D) at ({cos(30)}, 0.5);
\coordinate (w) at (-0.5,0.86);
\coordinate (a) at (-0.13,1.22);
\draw [red] (0.3,0.15) rectangle (0.6,0.4);
\draw (0.8,0)node[right]{$f$}->(0.6,0.15);
\draw [green] (0.43,0.5) circle [radius={cos(30)-0.43}];
\draw pic[draw=orange, <->,"$\omega$", angle eccentricity=1.5] {angle = D--w--a};
\draw [domain=30:150] plot ({cos(\x)}, {sin(\x)});
\draw [thin, dashed] (S) -- (D);
\draw [snake it][->](S) -- (w);
\draw [snake it][->](w) -- (D);
\draw [thin,dashed] (w)node[above]{$\vx$}--(a);
\node at (-0.95,0.48) {$\vs$};
\node at (0.95,0.5) {$\vd$};
\node at (-0.75,0.75) {$E$};
\node at (0,0.825) {$E'$};
\node at (0,-0.065) {$C_2$};
\node at (0,1.06) {$C_1$};
\draw [domain=-0.866:0.866] plot(\x, {-sqrt(1-pow(\x,2))+1});
\end{tikzpicture}
\caption{Part of a toric section $T=C_1\cup C_2$ with points of
self-intersection at source and detector points $\vs$ and $\vd$
respectively. The incoming photons (illustrated by wavy lines) have
initial energy $E$ and scatter at a fixed angle $\omega<\pi/2$ along
scattering sites $\vx\in T$. The resulting (scattered) photon energy
is $E'$ as in equation \eqref{equ1}. The electron density $f$ (the red
rectangle) is supported within the green circle (the unit ball, see
figure \ref{fig1}).} \label{fig0}
\end{figure}
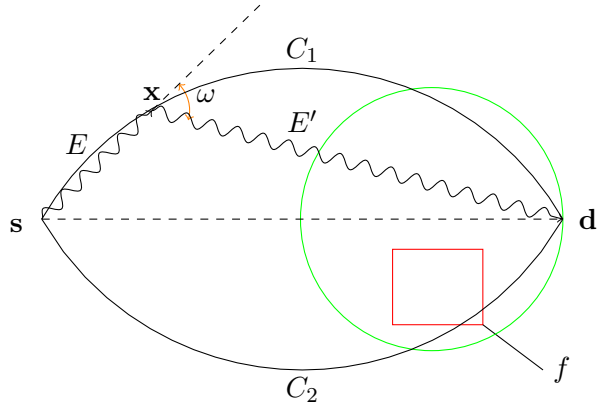where $E'$ is the scattered energy, $E$ is 
the initial energy, $\omega$ is the scattering angle and $E_0$ denotes
the electron rest energy. If the source is monochromatic ($E$ is
fixed) and we can measure the scattered energy $E'$, then the
scattering angle $\omega$ of the interaction is determined by equation
(\ref{equ1}). This implies that the locus of Compton scatterers in the
plane is a toric section $T=C_1\cup C_2$ (the union of
two intersecting circular arcs). See figure
\ref{fig0}. Hence we model the scattered intensity collected at the
detector $\vd$ with scattering angle $\omega$ (determined by the
scattering energy $E'$ in \eqref{equ1} and which determines the radius
$r$ of the circular arcs in figure \ref{fig1}) as integrals of the
electron charge density $f$ (represented by a real valued function)
over toric sections $T$. This is the idea behind two dimensional
Compton scattering tomography \cite{NT,norton,pal,me}.  Note that the
larger circular arcs of figure \ref{fig1} (which make up the majority
of the circle circumference) do not intersect the scanning region, and
hence we can consider integrals over whole toric sections (not just
the part of $T$ depicted in figure \ref{fig0}). In three dimensions,
the surface of scatterers is described by the surface of revolution of
a toric section about its central axis, namely a spindle torus. In
\cite{webberholman, me2} the inversion and microlocal aspects of a
spindle torus integral transform are considered.  In \cite{Rigaud2017}
Rigaud considers a related Compton model with attenuation, and Rigaud
and Hahn develop and analyze a clever contour reconstruction method
for a 3-d model \cite{RigaudHahn2018}.

The set of toric sections whose tips (the points of intersection of
$C_1$ and $C_2$) lie on two circles (as in figure \ref{fig1}) is three
dimensional. Indeed we can vary a source and detector coordinate on
$S^1\times S^1$ and the radius of the circles $r$. In this paper we
consider the two dimensional subset of toric sections whose central
axis (the line through the points of intersection of $C_1$ and $C_2$)
intersects the origin. This can be parametrized by a rotation about
the origin ($\theta\in S^1$) and the radius $r\geq 2$, as we shall see
later in section \ref{tsec}.

In \cite{me} the RTT geometry is considered and the scattered
intensity is approximated as a set of integrals over discs whose
boundaries intersect a given source point, and inversion techniques
and stability estimates are derived through an equivalence with the
Radon transform. Here we present a novel toric section transform
(which describes the scattered intensity exactly) and analyse its
stability from a microlocal standpoint. So far the results of Natterer
\cite{natterer} have been used to derive Sobolev space estimates for
the disc transform presented in \cite{me}, but the microlocal aspects
of the RTT geometry in Compton tomography are less well-studied. We
aim to address this here. We explain the expected artefacts in a
reconstruction from toric section integral data through an analysis of
the canonical relation of a toric section transform, and injectivity
results are provided for $L^\infty$ functions inside the unit ball.
The expected artefacts are shown by simulations and are as predicted
by the theory. We also give reconstructions of two simulated test
phantoms with varying levels of added pseudo-random noise. In
\cite{me2} it is suggested to use a Total Variation (TV)
regularization technique to combat the artefacts in a three
dimensional Compton tomography problem. Here we show that we can
combat the non-local artefacts (due to the 2-1 nature of the canonical relation) present in the reconstruction effectively in two
dimensions using a discrete approach and a heuristic TV regularizer.
\begin{figure}[!h]
\centering
\begin{tikzpicture}[scale=4.0]
\node at (0,0.87) {$\vs$};
\node at (0.05,-0.9) {$\vd$};
\path ({-cos(30)},0.5) coordinate (S);
\path (0, -0.8) coordinate (D);
\path  (0.6,0) coordinate (w);
\path (0,0) coordinate (a);
\draw pic[draw=orange, <->,"$\omega$", angle eccentricity=1.35] {angle = a--w--D};
\draw [rotate= 90] (0,-0.6) circle [radius=1];
\draw  [rotate=90] (0,0.6) circle [radius=1];
\draw   [->,line width=1pt] (0,-0.4)--(0,0);
\draw    [rotate=90] [<-] (0,-0.6)--(0,0);
\node   at   (-0.05,0.05) {$\theta$};
\node  at  (0.3,0.04) {$s$};
\node at (0.25,-0.35) {$r$};
\draw [->,line width=1pt] (-1,-0.4)--(1,-0.4);
\draw [->,line width=1pt] (0,-1.4)--(0,0.6);
\draw [->] (0,0)--(-0.6,0);
\node at (-0.1,0.6) {$x_2$};
\node at (1,-0.5) {$x_1$};
\draw [green] (0,-0.4) circle [radius=0.4];
\draw [blue] (0,-0.4) circle [radius=1.2];
\draw [red] (-0.25,-0.6) rectangle (0.05,-0.3);
\draw (-0.4,-0.8)node[left]{$f$}->(-0.25,-0.6);
\node at (-0.7,0) {$\vc_1$};
\node at (0.7,0) {$\vc_2$};
\draw (0.6,0)--(0,-0.8);
\node at (-0.43,0.43) {$C_2$};
\node at (0.4,0.43) {$C_1$};
\end{tikzpicture}
\caption{Part of a toric section $T=C_1\cup C_2$ with axis of rotation
$\theta=(0,1)$, tube center offset $s=\sqrt{r^2-4}$ and tube radius
$r$. Here $\cos\omega=\frac{\sqrt{r^2-4}}{r}$. The coordinates $\vc_1$ and $\vc_2$ denote the centers of the circles which the arcs
  $C_1$ and $C_2$ lie on respectively.  The detector ring
  (green circle, radius 1, center $O$) is the scanning region, where
  the density $f$ (the red square) is supported. The source ring is the blue circle, which has radius 3 and center $O$.} \label{fig1}
\end{figure}
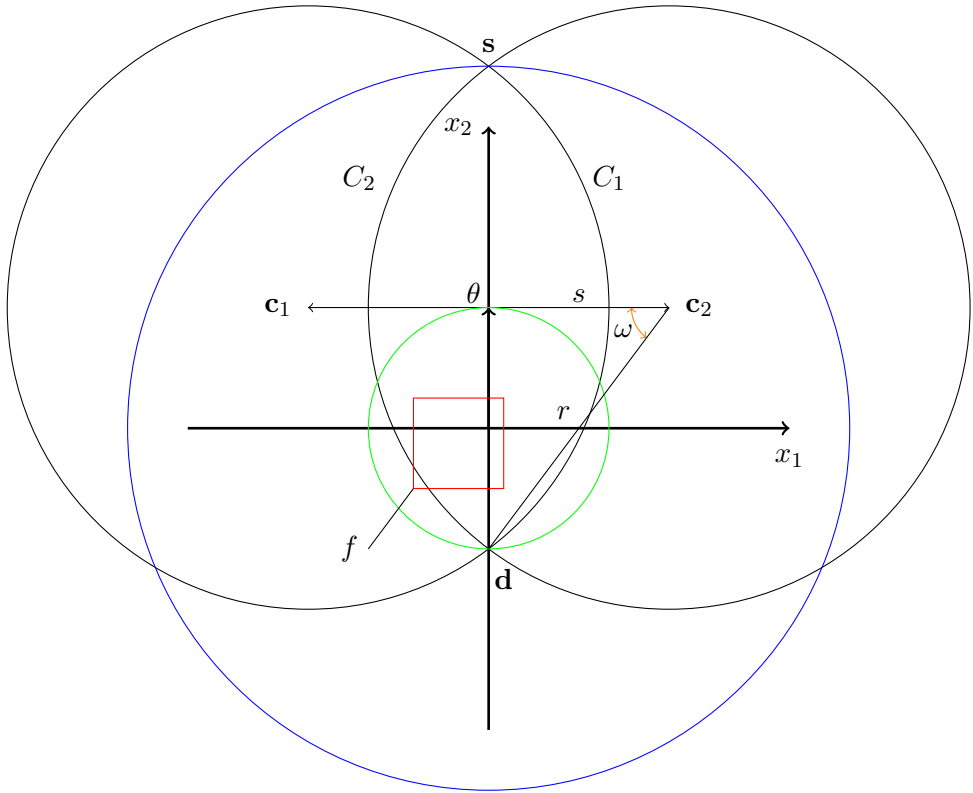
In section \ref{microsec} we recall some definitions and results on
Fourier Integral Operators (FIO's) and microlocal analysis before
introducing a new toric section transform in section \ref{tsec},
which describes the Compton scattered intensity collected by the
acquisition geometry in figure \ref{fig1}. Later in section \ref{tsec1}
we provide a novel microlocal analysis of the toric section transform
when considered as an FIO. Through an analysis of the canonical
relations of two circle transforms separately (whose sum is equivalent
to the toric section transform), we show that the canonical relation
of the toric section transform is 2--1 and provide explicit
expressions for the artefacts expected in a reconstruction from toric
section integral data.

In section \ref{tsec2} we prove the injectivity of the toric section
transform on the set of $L^\infty$ functions in the unit ball. This
uses a similar parameterization of circular arcs to Nguyen and Truong
in \cite{NT} and proves the injectivity by a decomposition into the
Fourier series components and using the ideas of Cormack
\cite{cormack}.

In section \ref{res}, we present a practical reconstruction algorithm for
the recovery of two dimensional densities from toric section integral data
and provide simulated reconstructions of two test phantoms (one simple and
one complex) with varying level of added pseudo-random noise. Here we use a
discrete approach. That is we discretize the toric section integral
operator (stored as a sparse matrix) on a pixel grid (assuming a piecewise
constant density) and use an iterative technique (e.g. a conjugate gradient
method) to solve the sparse set of linear equations described by the
discretized operator with regularization (e.g. Tikhonov or total
variation). We demonstrate the non-local artefacts in the reconstruction by an
application of the discretized normal operator ($A^TA$, where $A$ is the
discrete from of the toric section transform) to a delta function, and show
that the artefacts are exactly as predicted by the theory presented in
section \ref{tsec1} by a side by side comparison. We further show that we
can effectively combat the non-local reconstruction artefacts by applying the ``IRhtv" method of
\cite{IRtools} (see also \cite{AIRtools}).

\section{Microlocal definitions}\label{microsec}
We now provide  some definitions.
\begin{definition}[{\cite[Definition 7.1.1]{hormanderI}}]
For a function $f$ in the Schwartz space $S(\mathbb{R}^n)$ we define
the Fourier transform and its inverse 
as
\begin{equation}
\begin{split}
\mathcal{F}f(\xi)&=\int_{\mathbb{R}^n}e^{-ix\cdot\xi}f(x)\mathrm{d}x,\\
\mathcal{F}^{-1}f(x)&=(2\pi)^{-n}\int_{\mathbb{R}^n}e^{ix\cdot\xi}f(\xi)\mathrm{d}\xi.
\end{split}
\end{equation}
\end{definition}

We use the standard multi-index notation; let
$\alpha=(\alpha_1,\alpha_2,\dots,\alpha_n)\in \sparen{0,1,2,\dots}^n$
be a multi-index and $f$ a function on $\rn$, then $\partial^\alpha
f=\paren{\frac{\partial}{\partial
x_1}}^{\alpha_1}\paren{\frac{\partial}{\partial
x_2}}^{\alpha_2}\cdots\paren{\frac{\partial}{\partial x_n}}^{\alpha_n}
f$.

We identify cotangent spaces on Euclidean spaces with the underlying
Euclidean spaces so if $X$ is an open subset of $\rn$ and $(x,\xi)\in
X\times \rr^N$ then $T^*_{(x,\xi)}\paren{X\times \rr^N}$ is identified
with $\rn\times \rr^N$.  Under this identification, if
$\phi=\phi(x,\xi)$ for $(x,\xi)\in X\times \rr^N$ then
\[\begin{gathered}\dd_x \phi = \paren{\partxf{1}{\phi},
\partxf{2}{\phi}, \cdots, \partxf{n}{\phi} },\ \dd_\xi \phi =
\paren{\partxif{1}{\phi},\partxif{2}{\phi}, \cdots, \partxif{N}{\phi}
}\\ \text{ and }\ \dd\phi(x,\xi) = \paren{\dd_x \phi(x,\xi),\dd_\xi
\phi(x,\xi)}\in \rn\times\rr^N.\end{gathered}\]

\begin{definition}[{\cite[Definition 7.8.1]{hormanderI}}]
Let $X$ be an open subset of $\rn$ and let $m \in\mathbb{R}$. Then we define
$S^m(X\times\mathbb{R}^N)$ to
be the set of $a\in C^{\infty}(X\times \mathbb{R}^N)$ such that for
every compact set $K\subset X$ and all multi--indices $\alpha, \beta$
the bound
\[
\left|\partial^{\beta}_x\partial^{\alpha}_{\xi}a(x,\xi)\right|\leq
C_{\alpha,\beta,K}(1+|\xi|)^{m-|\alpha|},\ \ \ x\in K,\ \xi\in\mathbb{R}^n,
\]
holds for some constant $C_{K}$. The elements of $S^m$ are called
\emph{symbols} of order $m$.
\end{definition}

Note that these symbols are sometimes denoted $S^m_{1,0}$

\begin{definition}[{\cite[Definition
        21.2.15]{hormanderIII}}] \label{phasedef}
A function $\phi=\phi(x,\xi)\in
C^{\infty}(X\times\mathbb{R}^N\backslash 0)$ is a \emph{phase
function} if $\phi(x,\lambda\xi)=\lambda\phi(x,\xi)$, $\forall
\lambda>0$ and $\mathrm{d}\phi$ is nowhere zero. A phase function is
\emph{clean} if the critical set $\Sigma_\phi = \{ (x,\xi) \ : \
\mathrm{d}_\xi \phi(x,\xi) = 0 \}$ is a smooth manifold with tangent
space defined by $\mathrm{d} \paren{\mathrm{d}_\xi \phi}= 0$.
\end{definition}
\noindent By the implicit function theorem the requirement for a phase
function to be clean is satisfied if $\mathrm{d}\paren{\mathrm{d}_\xi
\phi}$ has constant rank.

\begin{definition}[{\cite[Definition 21.2.15]{hormanderIII} and
      \cite[Section 25.2]{hormander}}]\label{def:canon}
Let $X \subset \mathbb{R}^{n_x}$, $Y \subset \mathbb{R}^{n_y}$ be open
sets.  Let $\phi\in C^\infty\paren{X \times Y \times
\paren{\mathbb{R}^N\setminus 0}}$ be a clean phase function.  Then,
the \emph{critical set of $\phi$} is
\[\Sigma_\phi=\{(x,y,\xi)\in X\times Y\times\mathbb{R}^N\smo
: \dd_{\xi}\phi=0\}.\] The \emph{canonical relation parametrized by
$\phi$} is defined as
\bel{def:Cgenl}
 \begin{aligned}
 \Cc=&\sparen{
\paren{(y,\dd_y\phi(x,y,\xi)),(x,-\dd_x\phi(x,y,\xi))}:(x,y,\xi)\in
\Sigma_{\phi}},
\end{aligned}
\end{equation}
\end{definition}

\begin{definition}\label{FIOdef}
Let $X \subset \mathbb{R}^{n_x}$, $Y \subset \mathbb{R}^{n_y}$ be open
sets. A Fourier integral operator (FIO) of order $m + N/2 - (n_x +
n_y)/4$ is an operator $A:C^{\infty}_0(X)\to \mathcal{D}'(Y)$ with
Schwartz kernel given  by an oscillatory integral of the form
\begin{equation} \label{oscint}
K_A(x,y)=\int_{\mathbb{R}^N} e^{i\phi(x,y,\xi)}a(x,y,\xi) \mathrm{d}\xi,
\end{equation}
where $\phi$ is a clean phase function and $a \in S^m(X \times Y
\times \mathbb{R}^N)$ a symbol. The canonical relation of $A$ is the
canonical relation of $\phi$ defined in \eqref{def:Cgenl}.
\end{definition}

This is a simplified version of the definition of FIO in \cite[Section
2.4]{duist} or \cite[Section 25.2]{hormander} that is suitable for our
purposes since our phase functions are global.  For general
information about FIOs see \cite{duist, hormanderIII, hormander}.

\begin{definition}
\label{defproj} Let $\Cc\in T^*(Y\times X)$ be the canonical relation
associated to the FIO $A:\mathcal{E}'(X)\to \mathcal{D}'(Y)$. Then we
denote $\pi_L$ and $\pi_R$ to be the natural left- and
right-projections of $\Cc$, $\pi_L:\Cc\to T^*Y\backslash 0$ and $\pi_R
: \Cc\to T^*X\backslash 0$.
\end{definition}

 We have the following result from \cite{hormander}.
\begin{proposition}
\label{prop1}
Let $\dim(X)=\dim(Y)$. Then at any point in $\Cc$:
\begin{enumerate}[label=(\roman*)]
\item if one of $\pi_L$ or $\pi_R$ is a local diffeomorphism, then
$\Cc$ is a local canonical graph; 

\item if one of the projections $\pi_R$ or $\pi_L$ is singular at a
point in $\Cc$, then so is the other. The type of the singularity may
be different but both projections drop rank on the same set
\begin{equation}
\Sigma=\{(y,\eta; x,\xi)\in \Cc :
\det(\mathrm{d}\pi_L)=0\}=\{(y,\eta; x,\xi)\in \Cc : \det
(\mathrm{d}\pi_R)=0\}.
\end{equation}
\end{enumerate}
\end{proposition}

If a FIO $\Fc$ satisfies our next definition and $\Fc^t$ is its formal
adjoint, then $\Fc^t \Fc$ (or $\Fc^t \phi \Fc$ where $\phi\in \Dc(Y)$
if $\Fc$ and $\Fc^t$ cannot be composed) is a pseudodifferential
operator \cite{GS1977, quinto}.

\begin{definition}[Semi-global Bolker Assumption]\label{def:bolker} Let
$\Fc:\Ec'(X)\to \Dc'(Y)$ be a FIO with canonical relation $\Cc$ then
$\Fc$ (or $\Cc$) satisfies the \emph{semi-global Bolker Assumption} if the
natural projection $\pi_Y:\Cc\to T^*(Y)$ is an injective
immersion.\end{definition}




\section{A toric section transform} \label{tsec} In this section we
recall some notation and definitions and introduce a toric section
transform which models the intensity of scattered radiation described
by the acquisition geometry in figure \ref{fig1}. This section
contains our main theoretical results.  We describe microlocally the
expected artefacts in any \emph{backprojection} reconstruction
from toric section integral data (Theorem \ref{thm:main} and Remarks
\ref{rem:strength} and \ref{rem:what it means}).  In addition, we
prove the injectivity of the toric section transform using integral
equations techniques (Theorem \ref{thm:IntegralEqun} and Remark
\ref{rem:Cormack}).

For $r>0$, let $B_r$ be the open disk centered at the origin of radius
$r$ and let $B=B_1$ denote the open unit disk.  For $X$ an open subset
of $\mathbb{R}^n$, let $\mathcal{D}'(X)$ denote the vector space of
distributions on $X$, and let $\mathcal{E}'(X)$ denote the vector
space of distributions with compact support contained in $X$.

Let us parametrize points on the unit circle, $\theta\in S^1$ as
$\theta=\theta(\alpha)=(\cos\alpha,\sin\alpha)$, for $\alpha\in \otp$,
and let $\theta_{\alpha}=\frac{d\theta}{d\alpha}$ be the unit vector
$\pi/2$ radians counterclockwise (CCW) from $\theta$.  When the
choice of $\alpha$ is understood, then we will write $\theta$ for
$\theta(\alpha)$.

Let $(r,\alpha)\in Y:=(2,\infty)\times \otp$.  To define the toric
section, we first define two circular arcs and their centers.  For
$(r,\alpha)\in Y$ define \bel{defCs}
\begin{gathered}
s=\sqrt{r^2-4},
\quad
\vc_1=\vc_1(r,\alpha)=\theta(\alpha)+s\theta_{\alpha}(\alpha),
\quad  
\vc_2=\vc_2(r,\alpha)
=\theta(\alpha)-s\theta_{\alpha}(\alpha)
\\
C_1=C_1(r,\alpha)=\{\vx\in\mathbb{R}^2 :
\vx\cdot \theta_\alpha \leq0,\ |\vx-\vc_1(r,\alpha)|^2-r^2=0\},\notag\\
C_2=C_2(r,\alpha)=\{\vy\in\mathbb{R}^2 : \vy\cdot\theta_\alpha
\geq0,\
|\vy-\vc_2(r,\alpha)|^2-r^2=0\}.
\end{gathered}\ee
When the choice of $(r,\alpha)$ is understood, we will refer to the
arcs as $C_j$ and their centers as $\vc_j$ for $j\in \sparen{1,2}$.




The toric transform integrates functions on $B$ over the \emph{toric
sections}, $C_1(r,\alpha)\cup C_2(r,\alpha)$: let $f\in
C^{\infty}_0(B)$ represent the charge density in the plane.  Then, we
define the \emph{circle transforms} \bel{def:Tj} \mathcal{T}_1
f(r,\alpha)=\int_{C_1}f\mathrm{d}s,\ \ \ \ \ \ \ \
\mathcal{T}_2f(r,\alpha)=\int_{C_2}f\mathrm{d}s.
\end{equation}
and the \emph{toric section transform}
\bel{def:T}
\mathcal{T}f(r,\alpha)=\int_{C_1\cup
C_2}f\mathrm{d}s=\Tc_1(f)(r,\alpha)+\Tc_2(f)(r,\alpha)
\ee
where $\mathrm{d}s$ denotes the arc element on a circle.

 \begin{remark}\label{rem:T*} Let $j=1,2$.  The adjoint, $\Tc^t_j$, of
$\Tc_j$ is defined on distributions by duality.  For $g\in \Dc(Y)$ and
$\vx\in \rtwo\smo$, $\Tc^t g(\vx)$ is a weighted integral of $g$ over
all toric sections through $\vx$. Since there are no toric sections
intersecting points outside of $B_3$, we assume $\vx\in B_3$.  We also
note that no toric sections go through $\zero$--toric sections close
to $\zero$ have values of $r\approx \infty$.

Furthermore, for fixed $\vx\in B_3\setminus B$, the values of $\alpha$
such that $\vx\in C_j(r,\alpha)$ (for some $r$) is a proper
subinterval of $[0,2\pi]$.  

Since the set of toric sections is unbounded, $\Tc^t$ must be defined
on distributions of compact support.  

To deal with all of these inconveniences, we define a modified
adjoint.  Let $\vp:(2,\infty)\to\rr$ be smooth and with compact
support in $(2,M)$ for some $M>2$.  One can also assume $0\leq \vp\leq
1$ and $\vp=1$ on most of $(2,M)$.  We define the
\emph{cutoff-adjoint} $\Tc^*_j:\Dc'(Y)\to\Dc'(B_3)$. For $g\in
\Dc'(Y)$, \bel{def:T*} \Tc^*_j g=\Tc^t_j(\vp g),\qquad \Tc^* =
\Tc_1^*+\Tc_2^*.\ee 

Let $\rho_{\min}=M-\sqrt{M^2-3}$. Then, $\Tc^* g(\vx)=0$ for $\vx\in
B_{\rho_{\min}}\smo$. This is true because $\rho_{\min}$ is the
closest distance of the arcs $C_1(r,\alpha)$ and $C_2(r,\alpha)$ get
to the origin for all $(r,\alpha)\in (2,M)\times \otp$.  Therefore, we
define $\Tc^* g(0)=0$ and $\Tc^*g$ is smooth near $0$.  This also
means for $f\in \Ec'(B_3)$ that $\Tc^*\Tc f (\vx)=0$ if $\vx\in
B_{\rho_{\min}}$.
\end{remark}

In this section, we will study the microlocal properties
of $\Tc^*\Tc$. In Remark \ref{rem:strength}, we generalize our results
to a more general filtered backprojection. The main results of this
section are as follows.  Let $f\in \mathcal{E}'(B)$ have a singularity
(e.g., region boundary) at $\vw\in B$ in direction $\xi\in \rn\smo$,
with $\vw\cdot\xi\neq 0$, and let $\xi'=\xi/|\xi|$.  Our main theorem
(Theorem \ref{thm:main}) proves the existence of image artefacts
corresponding to $(\vw,\xi)$ in a reconstruction from $\Tc f$ data at
two points $\vx, \vy\in\mathbb{R}^2$. The expression for $\vy$ is
given explicitly by
\begin{equation}
    \vy=\nu[\theta_{\alpha},\theta]\begin{bmatrix}
    -\theta_{\alpha}^T \\
    \frac{2}{s}\theta_{\alpha}^T-\theta^T
    \end{bmatrix}\vw,
\end{equation}
where $r=\frac{|\vw|^2+3}{2(\vw\cdot\xi')}$ and $\theta$ satisfies
$$
\begin{pmatrix}
1 & s\\
-s & 1
\end{pmatrix}\theta=\vw-r\xi',
$$
and $\nu>0$ is chosen so that $\vy\in C_2$.  The artefact at
$\vy$ comes about when the singularity at $(\vw,\xi)$ is (co)normal to
a $C_1$ arc and is detected by $\Tc_1$ but backprojected by
$\Tc_2^*$.

The expression for $\vx$ is
given by
$$
    \vx=\frac{1}{\nu}[\theta_{\alpha},\theta]\begin{bmatrix}
    -\theta_{\alpha}^T \\
    -\frac{2}{s}\theta_{\alpha}^T-\theta^T
    \end{bmatrix}\vw,
$$
where $r=\frac{|\vw|^2+3}{2(\vw\cdot\xi')}$ and  $\theta$ satisfies
$$\begin{pmatrix}
1 & -s\\
s & 1
\end{pmatrix}\theta=\vw-r\xi',$$
and $\nu>0$ is chosen so that $\vx\in C_1$.  The artefact at
$\vx$ comes about when the singularity at $(\vw,\xi)$ is (co)normal to
a $C_2$ arc and is detected by $\Tc_2$ but backprojected by
$\Tc_1^*$.

 A visualization of the predicted image artefacts when $f$ is a delta
distribution is given in figure \ref{fig8}.

\subsection{Microlocal Properties of $\Tc_j$ and $\Tc$}\label{tsec1}

Since we do not consider the points of intersection of the arcs $C_1$
and $C_2$ (since distributions in the domain of $\Tc$, $\Ec'(B)$, are
supported away from them), we can consider the microlocal properties
of the circle transforms $\mathcal{T}_1$ and $\mathcal{T}_2$
separately.   Let $Y=\otp\times
(2,\infty)$.  When considering functions and distributions on $Y$, we
use the standard identification of $\otp$ with the unit circle $S^1$,
$\alpha\mapsto \theta(\alpha)=(\cos(\alpha),\sin(\alpha))$. 

We first show $\Tc_1$ and $\Tc_2$ are FIO.

\begin{proposition}\label{prop:FIO} $\Tc_1$ and $\Tc_2$ are both FIO
of order $-1/2$.  Their canonical relations are 
\begin{equation}\label{def:C1C2}
\begin{aligned}
\Cc_1=\Big\{\big(&r,\alpha,-2\sigma
\vx\cdot(\theta_{\alpha}-s\theta),-\frac{2\sigma
r}{s}(\vx\cdot\theta_{\alpha});
\vx,-2\sigma(\vx-\vc_1(r,\alpha))\big) : \\
&\quad (r,\alpha)\in Y, \sigma\in
\mathbb{R}\smo, \vx\in C_1(r,\alpha)\cap B\Big\}.\\
\Cc_2=\Big\{\big(&\alpha,r,-2\sigma
\vy\cdot(\theta_{\alpha}+s\theta),\frac{2\sigma
r}{s}(\vy\cdot\theta_{\alpha});
\vy,-2\sigma(\vy-\vc_2(r,\alpha))\big) : \\
&\quad (r,\alpha)\in Y, \sigma\in \mathbb{R}\backslash\{0\},\vy\in
C_2(r,\alpha)\cap B,
\Big\}.
\end{aligned}
\end{equation}
\end{proposition}

 For $j=1,2$, we let $\tCc_j$ be defined as $\Cc_j$ except that $\vx$
or $\vy$ is not restricted to be in $B$ and we let $\tCc=\tCc_1\cup
\tCc_2$.

\begin{proof}
  We briefly explain why $\Tc_2$ is a FIO and we calculate its canonical
  relation.  Let
  $Z=\sparen{(r,\alpha,\vy)\in Y\times B\st
    |\vy-\vc_2(r,\alpha)|^2-r^2=0}$.  From calculations in
  \cite{GS1977,quinto} the Schwartz kernel of $\Tc_2$ is integration over
  $Z$ and so the this Schwartz kernel is a Fourier integral distribution
  with phase function
  $\phi_2(\vy,r,\alpha,\sigma)=\sigma\left(|\vy-\vc_2(r,\alpha)|^2
    -r^2\right)$.  This is true because, for functions supported in $B$,
  $\Tc_2$ can be viewed as integrating on the full circle defined by
  $|\vy-\vc_2(r,\alpha)|^2-r^2=0$.

Using Definition \ref{def:canon} one sees that the canonical relation
of $\Tc_2$ is given by the expression in \eqref{def:C1C2}.  One can
easily check that the projections $\pi_L(\Cc_2)$ and $\pi_R(\Cc_2)$ do
not map to the zero section so $\Tc_2:\Ec'(B)\to \Dc'(Y)$
\cite{Ho1971}.   

The operator $\Tc_2$ is a Radon transform and therefore its symbol
  is of order zero (see, e.g., \cite{quinto}), so one can use the order
  calculation in Definition \ref{FIOdef} to show that the order of $\Tc_2$
  is $-1/2$.

In a similar way, one shows that  $\Tc_1$ is an FIO with phase function 
$\phi_1(\vx,r,\alpha,\sigma)=\sigma\left(|\vx-\vc_1(r,\alpha)|^2-r^2\right)$.\end{proof}

We now prove that each $\Tc_j$ satisfies the Bolker Assumption.

\begin{theorem}\label{thm:bolker}
For $j=1,2$, the left projection $\pi_L:\tCc_j\to T^*(Y)$ is an
injective immersion.  Therefore, $\pi_L:\Cc_j\to T^*(Y)$ is an
injective immersion and so $\Tc_j$ satisfies the semi-global Bolker
Assumption (Definition \ref{def:bolker}).

The operators $\Tc_i^*$ and $\Tc_j$ can be composed as FIO and
the compositions all have order $-1$.
\end{theorem}

\begin{proof}
We will prove this theorem for $\Tc_2$ and the proof for $\Tc_1$ is
completely analogous.  We first show that $\pi_L$ is an immersion.  

As noted above, if $\alpha$ is known, then we let
$\theta=\theta(\alpha)$ and $\theta_\alpha=(-\sin\alpha,\cos\alpha)$.
For bookkeeping reasons, if $\beta\in \otp$, the vector in $S^1$
corresponding to $\beta$ will be denoted $\psi =
(\cos\beta,\sin\beta)$ and we let
$\psi_{\beta}=(-\sin\beta,\cos\beta)$ be the unit vector $\pi/2$
radians CCW from $\psi$.  This allows us to parametrize points on
$C_2(r,\alpha)$ by \bel{def:y}\vy=\vy(r,\alpha,\beta)=\vc_2
+r\psi=\vc_2(r,\alpha)+r(\cos\beta,\sin\beta),\ee for $\beta$ in an
open interval containing $\otp$.  Then, \bel{coords:C2}(r,\alpha,
\beta, \sigma)\mapsto
\lambda_2(r,\alpha,\beta,\sigma):=\paren{r,\alpha,\sigma \dd_\alpha
\phi_2,\sigma\dd_r \phi_2;\vy(r,\alpha,\beta),
-\sigma\dd_\vy\phi_2}\in \Cc_2\ee gives coordinates on the canonical
relation $\Cc_2$.  Using these coordinates and after simplification,
the map $\pi_L$ is given by
\begin{equation}
\pi_L(\lambda(r,\alpha,\beta,\sigma))=\left(r,\alpha, -2\sigma
r\psi\cdot (\theta_{\alpha}+s\theta), \frac{2\sigma
r}{s}(-s+r\psi\cdot \theta_{\alpha})\right)
\end{equation}
and
\begin{equation}
D\pi_L = 
 \begin{pmatrix}
  1 & 0 & 0&0\\
  0 & 1 & 0&0\\
  a_{3,1}  & a_{3,2}  & -2\sigma r\psi_{\beta}\cdot
  (\theta_{\alpha}+s\theta)  
& -2 r\psi\cdot (\theta_{\alpha}+s\theta) 
\\
  a_{4,1} & a_{4,2} 

& \frac{2\sigma r^2}{s}(\psi_{\beta}\cdot \theta_{\alpha})  

& \frac{2 r^2}{s}\left(-\frac{s}{r}+\psi\cdot \theta_{\alpha}\right) 
 \end{pmatrix}.
 \end{equation}
It follows that
\begin{equation}
\begin{split}
\det D\pi_L&=-\frac{4r^3\sigma}{s}\det\begin{pmatrix}
\psi_{\beta}\cdot (\theta_{\alpha}+s\theta) 
&  \psi\cdot (\theta_{\alpha}+s\theta) 
\\
\psi_{\beta}\cdot \theta_{\alpha}  
& -\frac{s}{r}+\psi\cdot \theta_{\alpha} 
 \end{pmatrix}\\
 &=-\frac{4r^3\sigma}{s}\left(-\frac{s}{r}\psi_{\beta}\cdot
(\theta_{\alpha}+s\theta)+(\psi\cdot \theta_{\alpha})(\psi_{\beta}\cdot
\theta_{\alpha}+s\psi_{\beta}\cdot\theta)-(\psi_{\beta}\cdot
\theta_{\alpha})(\psi\cdot \theta_{\alpha}+s\psi\cdot\theta)\right)\\
 &=-4r^3\sigma\left(-\frac{1}{r}\psi_{\beta}\cdot
(\theta_{\alpha}+s\theta)+(\psi\cdot
\theta_{\alpha})(\psi_{\beta}\cdot\theta)-(\psi_{\beta}\cdot
\theta_{\alpha})(\psi\cdot\theta)\right)\\
 &=4r^3\sigma\left(\frac{1}{r}(\psi\cdot \vc_2)+\left((\psi\cdot
\theta_{\alpha})^2+(\psi\cdot\theta)^2\right)\right)\\
 &=4r^3\sigma\left(\frac{1}{r}(\psi\cdot \vc_2)+1\right), 
 \end{split}
\end{equation}
where to go from step 3 to 4 above we have used the identities
$\psi_{\beta}\cdot \theta_{\alpha}=\psi\cdot\theta$ and
$\psi_{\beta}\cdot \theta=-\psi\cdot \theta_{\alpha}$. Let us assume
$\det D\pi_L=0$. Then $\psi\cdot \vc_2 =-r$. But $|\psi\cdot
\vc_2|\leq |\vc_2|=\sqrt{r^2-3}<r$ and we have a contradiction.  Note that
this contradiction holds for all $\vy\in C_2(r,\alpha)$, not just those in
$B$.  Therefore, the map $\pi_L:\tCc_2\to T^*(Y)$ is an immersion.

We next show the injectivity of the left projection $\pi_L$ through an
analysis of the canonical relations of $\mathcal{T}_2$. Let
$(\ra,\eta)\in \pi_L(\Cc_2)$ and $\vy_1$ and $\vy_2$ be two points in
$C_2$ and $\xi$ and $\txi$ in $\rtwo\smo$ such that $(\ra,\eta;
\vy_1,\xi)$ and $(\ra,\eta; \vy_2,\txi)$ are both in $\Cc_2$. We show
$(\vy_1,\xi)=(\vy_2,\txi)$.  By equating the terms for $\eta$ in the
expression for $\Cc_2$, \eqref{def:C1C2}, one sees, for some
$\sigma_1$ and $\sigma_2$, that 
\begin{equation}\label{eta}
\eta = \begin{pmatrix}
-2\sigma_1\vy_1\cdot(\theta_{\alpha}+s\theta) \\
\frac{2\sigma_1 r}{s}(\vy_1\cdot\theta_{\alpha}) 
\end{pmatrix}=
\begin{pmatrix}
-2\sigma_2\vy_2\cdot(\theta_{\alpha}+s\theta) \\
\frac{2\sigma_2 r}{s}(\vy_2\cdot\theta_{\alpha})
\end{pmatrix}
\end{equation}
where $s=\sqrt{r^2-4}$.  Since $\vy_j\cdot \ta<0$, the bottom equation
in \eqref{eta} shows that $\nu=\sigma_1/\sigma_2>0$. In addition,
\begin{equation}
\frac{2\sigma_1r}{s}(\vy_1\cdot\theta_{\alpha})=\frac{2\sigma_2r}{s}(\vy_2\cdot\theta_{\alpha})\implies
(\sigma_1\vy_1-\sigma_2\vy_2)\cdot \theta_{\alpha}=0
\end{equation}
and
\begin{equation}
-2\sigma_1\vy_1\cdot(\theta_{\alpha}+s\theta)=-2\sigma_2\vy_2\cdot(\theta_{\alpha}+s\theta)\implies
(\sigma_1\vy_1-\sigma_2\vy_2)\cdot \theta=0.
\end{equation}
Hence $\sigma_1\vy_1-\sigma_2\vy_2=0$ or $\vy_2 =
\nu\vy_1$ where $\nu=\frac{\sigma_1}{\sigma_2}>0$. Given that any ray
through origin intersects the curve $ C_2$ at most once and
$\vy_1,\vy_2\in C_2$, it follows that $\sigma_1=\sigma_2$ and
$\vy_1=\vy_2$.  This finishes the proof for $\Tc_2$.  Note that this
proof is valid for any $\vy_1$ and $\vy_2$ in $C_2$, not just for
those in $B$.  In other words, $\pi_L:\tCc_2\to T^*(Y)$ is also
injective, so $\pi_L:\tCc_2\to T^*(Y)$ is an injective immersion.

As already noted, the proof for $\Tc_1$ is similar, and it uses the
following coordinate maps 
\begin{gather}\label{def:x}\vx=\vx(r,\alpha,\beta)=\vc_1
+r\psi=\vc_1(r,\alpha)+r(\cos\beta,\sin\beta),\ \beta\in \otp,\\
\label{coords:C1}(r,\alpha, \beta, \sigma)\mapsto
\lambda_1(r,\alpha,\beta,\sigma):=\paren{r,\alpha,\sigma \dd_\alpha
\phi_1,\sigma \dd_r \phi_1;\vx(r,\alpha,\beta),
-\sigma\dd_\vx\phi_1}\in \Cc_1,\end{gather} however in this case,
$\beta$ is in an open interval containing $[-\pi,\pi]$.

Since $\Tc_i$ and its dual are of order $-1/2$ and have
canonical relations that are local canonical graphs (as they satisfy
the Bolker Assumption), all compositions $\Tc_i^*\Tc_j$ are FIO of
order $-1$ \cite{Ho1971}.
\end{proof}

Let $\Cc = \Cc_1\cup \Cc_2$.  Because $\Cc_1\cap \Cc_2=\emptyset$
above $B$, $\Cc$ is an embedded Lagrangian manifold and since $\Tc =
\Tc_1+\Tc_2$, $\Tc$ is a FIO with canonical relation $\Cc$.  We now
have our main theorem which shows that the canonical relation
$\mathcal{C}$ is 2--1 in a specific sense.  We give explicit
expressions for the expected artefacts in a reconstruction using
$\Tc^*\Tc$ that are caused by this 2--1 map.

\begin{theorem}\label{thm:main}
The projection $\pi_L:\Cc\to T^*(Y)$ is two-to-one in the following
sense.  Let $\lambda=(r,\alpha,\eta)\in \pi_L(\Cc)$.  Then, there is
at least one point $(\vw,\xi)\in B\times \paren{\rtwo\smo}$ such that
$\lambda= \pi_L(\lambda,(\vw,\xi))$.  Necessarily, $\vw$ is either in $
C_1(r,\alpha)$ or in $C_2(r,\alpha)$.  Assume $\vw\in C_1$.  Then,
there is a $\vy\in C_2$ and $\txi\in \rtwo\smo$ such that $\lambda =
\pi_L(\lambda,(\vy,\txi))$.  The point $\vy$ is given by
\eqref{yC2}.  If $\vw\in C_2$, then its corresponding point in $C_1$
is given by \eqref{xC1}.

Let $\Tc^*$ be the modified dual operator in Remark \ref{rem:T*}.  The
canonical relation of $\mathcal{T}^*\mathcal{T}$ is of the form
$\Delta\cup \Lambda_1\cup\Lambda_2$, where $\Delta$ is the diagonal in
$T^*X\times T^*X$ and $\Lambda_1=\tCc_1^t\circ \Cc_2$ and
$\Lambda_2=\tCc_2^t\circ \Cc_1$ are associated to reconstruction
artefacts.  

Let $f$ be a distribution supported in $B$.  If $(\vw,\xi)\in \WF(f)$
and $\xi\cdot\vw\neq 0$, then two artefacts can be generated in
$\Tc^*\Tc f$ associated with $(\vw,\xi)$ (see remark
\ref{rem:elliptic}).
 The base point of the one generated by
$\Lambda_1$ is given by \eqref{art1} where $r$ is defined by
$\eqref{def:r}$ and $\alpha$ is solved from \eqref{orthog2} and the
base point of the artefact caused by $\Lambda_2$ is given by
\eqref{art2} where $r$ is defined by $\eqref{def:r}$ and $\alpha$ is
given by solving \eqref{orthog1}.
\end{theorem}

Artifacts occurs naturally in several other types of tomography,
such as in limited data X-ray CT \cite{BFJQ}.  The artifacts in this
Compton CT problem are similar to the left-right ambiguity in
synthetic aperture radar (SAR) \cite{homan, NC2004, SU:SAR2013}
because they are both come from the backprojection.  However, the
left-right artifacts in SAR (a mirror-image artifact appearing on the
opposite side of the flight path to an object on the ground) is
geometrically easier to characterize than the artifacts caused by the
$\Lambda_j$ given in Theorem \ref{thm:main}.

In both cases, if one could take only half of the data (e.g., in
Compton CT, only $\Tc_1$, or in SAR using side-looking radar) then one
would not have artifacts.  However, the authors are not aware of any
way reliably to  obtain only the data over $C_1$ (or only $C_2$) in the desired application in airport baggage screening (i.e. in the machine geometry of figure \ref{fig1}).

 \begin{remark}\label{rem:elliptic} In theorem \ref{thm:main},
we note artefacts \emph{can} occur, and we now discuss this more
carefully.  The backprojection reconstruction is made of four terms,
$\Tc^*\Tc = \Tc_1^*\Tc_1+ \Tc_2^*\Tc_2+ \Tc_1^*\Tc_2+ \Tc_2^*\Tc_1$,
and we first analyze the individual compositions.

 If $(\vx,\xi)$ is (co)normal to a circle $C_j(r,\alpha)$ with $r\in
\supp(\vp)$, then this singularity is visible in $\vp\Tc_j$ because
the cutoff $\vp$ is nonzero near $r$ and $\Tc_j$, is elliptic.
Therefore, the singularity will appear in the composition
$\Tc_j^*\Tc_j$, and any artefact caused by $\Tc^*_i\Tc_j$ when $i\neq
j$ will also appear.

On the other hand, if $(\vx,\xi)$ is (co)normal to a circle
$C_j(r,\alpha)$ with $r\notin \supp(\vp)$, then this singularity is
smoothed by $\vp\Tc_j$ because the cutoff $\vp$ is zero near $r$, and
the singularity will not appear in the composition $\Tc_j^*\Tc_j$, and
no artefact will be created by $\Tc^*_i\Tc_j$ when $i\neq j$.

However, artefacts and visible singularities can cancel each other
because $\Tc^*\Tc$ is the sum of four terms of the forms above.

\end{remark}

Our next remark describes the strength in Sobolev scale of the
artefacts and generalizes our theorem for \emph{filtered}
backprojection.

\begin{remark}\label{rem:strength} The artefacts caused by a
singularity of $f$ are as strong as the reconstruction of that
singularity. 

  The visible singularities come from the compositions $\Tc_1^*\Tc_1$
and $\Tc_2^*\Tc_2$ since these are pseudodifferential operators of
order $-1$. The artefacts come from the ``cross'' compositions
$\Tc_2^* \Tc_1$ and $\Tc_1^* \Tc_2$, and they are FIO of order $-1$.
Therefore, since the terms that preserve the real singularities of
$f$, $\Tc_i^*\Tc_i$, $i=1,2$, are also of order $-1$, $\Tc^*\Tc$
smooths each singularity of $f$ by one order in Sobolev norm
\emph{and} the compositions $\Tc^*_i\Tc_j$ for $i\neq j$ create
artefacts from that singularity that are also one order smoother than
that singularity.  

    Second, our results are valid, not only for the normal operator
$\Tc^*\Tc$ but for any \emph{filtered} backprojection method $\Tc^*P
\Tc$ where $P$ is a pseudodifferential operator.  This is true since
pseudodifferential operators have canonical relation $\Delta$ and they
do not move singularities, so our microlocal calculations are the
same.  If $P$ has order $k$, then $\Tc^*P\Tc$ smooths each singularity
of $f$ by order $-(k-1)$ in Sobolev norm and creates an artefact from
that singularity that is also $-(k-1)$ orders smoother.
\end{remark}

\begin{proof} 
Let $(\ra,\eta)\in \pi_L(\Cc)$, then there is an $(\vw,\xi)\in B\times
\paren{\rtwo\smo}$ such that $(\ra,\eta;\vw,\xi)\in \Cc$.  Either
$\vw\in C_1(\ra)$ or $\vw\in C_2(\ra)$, and this is determined by
$(\ra)$.  At the end of this part of the proof, we will outline what
to do if $\vw\in C_2(\ra)$.

We assume $\vw\in C_1(\ra)$ and, for this part of the proof--in which
$\vw\in C_1$--we let $\vx = \vw$.  Assume there is another point in $
\tCc$ that maps to $(\ra,\eta)$ under $\pi_L$.  That point must be on
$\tCc_2$ and it must be unique since $\pi_L:\tCc_j\to \Tsty$ is
injective for $j=1,2$ by Theorem \ref{thm:bolker}.  Let $(\vy,\txi)$
be chosen so $\vy\in C_2(\ra)$ and $(\ra,\eta;\vy,\txi)$ is the
preimage in $\tCc_2$ of $(\ra,\eta)$.  Comparing the $\eta$ term of
the expressions \eqref{def:C1C2} for $\Cc_1$ and $\Cc_2$, we see there
are numbers $\sigma_1$ and $\sigma_2$ such that 
\begin{equation}
\label{equ3}
\eta = \begin{pmatrix}
-2\sigma_1\vx\cdot(\theta_{\alpha}-s\theta) \\
-\frac{2\sigma_1 r}{s}(\vx\cdot\theta_{\alpha})
\end{pmatrix}=\begin{pmatrix}
-2\sigma_2\vy\cdot(\theta_{\alpha}+s\theta) \\
\frac{2\sigma_2 r}{s}(\vy\cdot\theta_{\alpha}).
\end{pmatrix},
\end{equation}
This implies that
$\sigma_1(\vx\cdot \theta_\alpha) = -\sigma_2(\vy\cdot\theta_\alpha)$.
Since $\vx\cdot \ta$ and $\vy\cdot \ta$ have opposite signs, $\sigma_1$ and
$\sigma_2$ have the same sign.  Let $\nu = \sigma_1/\sigma_2$, then $\nu>0$
and if we solve \eqref{equ3} for $\vy$, we see \bel{y from x} \vy =
\nu\paren{(-\vx\cdot \ta)\ta + \paren{\frac{2}{s} \vx\cdot \ta - \vx\cdot
    \theta}\theta} \ \text{ for some $\nu>0$.}\ee 
Equivalently we can write the above as
\begin{equation}
\label{yC2}
    \vy=\nu[\theta_{\alpha},\theta]\begin{bmatrix}
    -\theta_{\alpha}^T \\
    \frac{2}{s}\theta_{\alpha}^T-\theta^T
    \end{bmatrix}\vx,\ \ \ (C_1\to C_2).
\end{equation}
Given $\ra$ and $\vx$, this equation describes the point $\vy$ that is
the base point of the preimage in $\tCc_2$ of $(\ra,\eta)$.

Equation \eqref{yC2} for arbitrary $\nu>0$ describes a ray starting at
$\zero$.  Because the circle containing $C_2(\ra)$ encloses $\zero$, this
ray intersects the circle at a unique point.  Since any point $\vy'$ on
this ray satisfies $\vy'\cdot \ta<0$, the unique point on the circle is on
$C_2(\ra)$.  If $\vw=\vx\in C_1$, then this proves that $\pi_L$ is
two-to-one
as described in the theorem.

To prove the statement about $\pi_L$ being two-to-one if the point
$\vw$ at the start of the proof is in $C_2(\ra)$ then one goes through
the same proof but solves for $\vx$ in terms of $\vy$ and replace
$\vy$ by $ \vw$ in
\eqref{equ3} to get 
\begin{equation}
\label{xC1}
    \vx=\frac{1}{\nu}[\theta_{\alpha},\theta]\begin{bmatrix}
    -\theta_{\alpha}^T \\
    -\frac{2}{s}\theta_{\alpha}^T-\theta^T
    \end{bmatrix}\vw,\ \ \ (C_2\to C_1).
\end{equation}
Given $\ra$ and $\vw$, this equation describes the point $\vx$ that is
the base point of the preimage in $\tCc_1$ of $(\ra,\eta)$.

To describe explicitly the artefacts which occur due to an application
of the normal operator $\mathcal{T}^*\mathcal{T}$, let us consider the
canonical relation $\tCc^t\circ \Cc$. We have the expansion
\begin{equation}
\begin{split}
\tCc^t\circ \Cc&=(\tCc_1\cup\Cc_2)^t\circ (\tCc_1\cup\Cc_2)\\
&=(\tCc_1^t\cup\Cc_1)\cup (\tCc_2^t\cup\Cc_2)\cup (\tCc_1^t\cup\Cc_2)\cup
(\tCc_2^t\cup\Cc_1)\\
&\subset \Delta \cup \Lambda_1\cup\Lambda_2,
\end{split}
\end{equation}
where $\Lambda_1=\tCc_1^t\cup\Cc_2$ and $\Lambda_2=\tCc_2^t\cup\Cc_1$.
Note that $\tCc_j^t\circ \Cc_j\subset \Delta$ for $j=1,2$ because
$\tCc_j$ satisfies the Bolker Assumption.

Let $(\vw,\xi)\in T^*(B)$ be such
that $\vw\cdot\xi\neq 0$ and let $\xi' = \xi/\abs{\xi}$. We now
calculate the
$(r,\theta)$ for which the circular arc $C_1$ intersects $\vw$ normal to
$\xi$, explicitly in terms of $(\vw,\xi)$. For $\vw\in C_1$ we know
$\vc_1=\vw-r\xi'$. Therefore $$ |\vc_1|^2=r^2-3=|\vw-r\xi'|^2=|\vw|^2-2r
\vw\cdot\xi' +r^2
$$
and it follows that 
\bel{def:r}r=\frac{|\vw|^2+3}{2(\vw\cdot\xi')}.\ee
Also, to get $(r,\theta)$ explicitly in terms of $(\vw,\xi')$,
\bel{orthog1}
\begin{pmatrix}
1 & s\\
-s & 1
\end{pmatrix}\theta=\vw-r\xi'.
\ee
To check that
$\theta$ is a unit vector, note that
\[
|\theta|=\frac{1}{1+s^2}\left|
\begin{pmatrix}
1 & -s\\
s & 1
\end{pmatrix}(\vx-r\xi')\right|=\frac{|\vw-r\xi'|\sqrt{1+s^2}}{1+s^2}=1,
\]
as $|\vw-r\xi'|=|\vc_1|=\sqrt{1+s^2}$. Once $(r,\theta)$ are known, the
artefact $\vy$ induced by $\Lambda_2$ is given by equation
\eqref{yC2}
\begin{equation}\label{art2}
    \vy=\nu[\theta_{\alpha},\theta]\begin{bmatrix}
    -\theta_{\alpha}^T \\
    \frac{2}{s}\theta_{\alpha}^T-\theta^T
    \end{bmatrix}\vw
\end{equation}
where $\nu>0$ is such that $\vy\in C_2$.  This point $\vy$ is the base
point of the artefact corresponding to $(\vw,\xi)$ that is added by
$\Lambda_1$.

Similarly we can express the $(r,\theta)$ for which the circular arc $C_2$
intersects $\vw$ normal to $\xi$, explicitly in terms of $(\vw,\xi)$.
When $\vw\in C_2$, we know $\vc_2=\vw-r\xi'$. Hence the calculation for
$r$ is the same as \eqref{def:r} and
\bel{orthog2}
\begin{pmatrix}
1 & -s\\
s & 1
\end{pmatrix}\theta=\vw-r\xi',
\ee
and hence the artefact $\vx$ induced by $\Lambda_1$ is given by
\eqref{xC1}
\bel{art1}
    \vx=\frac{1}{\nu}[\theta_{\alpha},\theta]\begin{bmatrix}
    -\theta_{\alpha}^T \\
    -\frac{2}{s}\theta_{\alpha}^T-\theta^T
    \end{bmatrix}\vw,
\ee
where $\nu$ is chosen so $\vx\in C_1$. Then, $\vx$ is the base point of
the artefact in $\Lambda_1$ caused by $(\vw,\xi)$.
\end{proof}

\begin{remark}\label{rem:what it means} Theorem
\ref{thm:main} proves that $\Cc$ is 2-1 everywhere above $B$, and
equations \eqref{art2} and \eqref{art1} provide expressions for the
pairs $\vx,\vy$ whose image under $\Cc$ is the same. Intuitively we
can think of this as an inherent ``confusion" in the data $\Tc f$ as
to where the ``true" singularities (e.g., object boundaries or
contours) in $f$ lie (and in what directions). To give more detail,
let $f$ have a singularity at $\vw$ in direction $\xi$.  The
singularity at $\vw$ is detected in the data $\Tc f$ when the circular
arc $C_j$ (for some $j=1, 2$) intersects $\vw$ normal to $\xi$. Such a
$C_j$ always exists by Theorem \ref{thm:main} (see the expressions for
$(r,\theta)$ in terms of $(\vw,\xi)$), and hence the singularity at
$\vw$ is resolved. However, due to the 2-1 nature of $\Cc$, we only
have sufficient information to say that the true singularity lies at
$\vw$ \emph{or} some $\vx,\vy$ (as in equations \eqref{art2} and
\eqref{art1}). Hence we see image artefacts in the reconstruction at
$\vx$ (for $(\vw,\xi)\in N^*C_2$) and $\vy$ (for $(\vw,\xi)\in
N^*C_1$), and the artefacts appear as ``additional" (unwanted) image
singularities on one-dimensional manifolds (see figure \ref{fig8}).
\end{remark}

\subsection{Injectivity}\label{tsec2} Here we prove the injectivity of
the toric section transform $\mathcal{T}$ on $L_c^\infty(B)$,
$L^\infty$ functions of compact support in $B$.  We write points in
$\rtwo$ in polar coordinates $(\rho,\alpha)\mapsto \rho
\theta(\alpha)= \rho(\cos(\alpha),\sin(\alpha))$.  For an integrable
function $F(\rho,\alpha)$ and $l\in \zz$, we define the
\emph{${l}^{\text{th}}$ polar Fourier coefficient of $f$} to be 
\[F_l(\rho)=\frac{1}{2\pi} \int_{\alpha = 0}^{2\pi}
F(\rho,\alpha)e^{-il \alpha}\,d\alpha. \]

Let $t=\sqrt{r^2-3}$ and let $\alpha(t)=\cos^{-1}\frac{1}{t}$. Then
we can parametrize the set of points on the toric section in polar
coordinates
\begin{equation}
\begin{split}
\rho&=\sqrt{t^2\cos^2\varphi+3}-t\cos\varphi,\ \ \ \ -\alpha(t)\leq
\varphi\leq\alpha(t),\ t\geq1\\
\theta&=\alpha+\alpha(t)+\varphi,\ \ \ \text{or}\ \ \ \theta=\alpha-\alpha(t)+\varphi,\ \ \ \ 0\leq\alpha\leq2\pi
\end{split}
\end{equation}
and it follows that
\begin{equation}\label{ttrans}\begin{aligned}
\mathcal{T}f(t,\alpha)=\int_{-\alpha(t)}^{\alpha(t)}\sqrt{\rho^2+\left(\frac{\partial
    \rho}{\partial\varphi}\right)^2}
\large[&F(\rho,\alpha+\alpha(t)+\varphi)\\
&\ +F(\rho,\alpha-\alpha(t)+\varphi)\large]\mid_{\rho=\sqrt{t^2\cos^2\varphi+3}-t\cos\varphi}\mathrm{d}\varphi,\end{aligned}
\end{equation}
where $F(\rho,\alpha)=f(\rho\theta(\alpha))$ is the polar form of $f$. We
now have our second main theorem which follows using similar ideas to
Cormack's \cite{cormack}.

\begin{theorem}\label{thm:IntegralEqun}
The toric section transform $\mathcal{T} : L^\infty_c(B)\to
L^\infty(Y)$, where $Y=(2,\infty)\times \otp$, is injective.
\end{theorem}
\begin{proof}
After exploiting the rotational invariance of the transform (\ref{ttrans})
we have
\begin{equation}
\label{equ2.1}
\paren{\mathcal{T}f}_l(t)=T_{\abs{l}}\left(\frac{1}{t}\right)\int_{-\alpha(t)}^{\alpha(t)}\sqrt{\rho^2+\left(\frac{\partial
    \rho}{\partial\varphi}\right)^2}F_l(\rho)e^{-il\varphi}\mid_{\rho=\sqrt{t^2\cos^2\varphi+3}-t\cos\varphi}\mathrm{d}\varphi,
\end{equation}
where
\begin{equation}
\paren{\mathcal{T}f}_l(t)=\frac{1}{2\pi}\int_{0}^{2\pi}\mathcal{T}f(t,\alpha)e^{-il\alpha}\mathrm{d}\alpha,
\end{equation}
and $T_{\abs{l}}$ is Chebyshev polynomial of
the first kind of order $\abs{l}$.

The arc length measure on the circle is
\begin{equation}
\mathrm{d}s=\sqrt{\rho^2+\left(\frac{\partial
    \rho}{\partial\varphi}\right)^2}\mathrm{d}\varphi=\rho\sqrt{\frac{t^2+3}{t^2\cos^2\varphi+3}}\mathrm{d}\varphi=r\left(1-\frac{t\cos\varphi}{\sqrt{t^2\cos^2\varphi+3}}\right)\mathrm{d}\varphi
\end{equation}
and using the symmetry of equation (\ref{equ2.1}) in $\varphi$ about
$\varphi=0$ we have
\begin{equation}
\begin{split}
\label{equ2}
\frac{\paren{\mathcal{T}f}_l(t)}{4r}&=T_{\abs{l}}\left(\frac{1}{t}\right)\int_{0}^{\alpha(t)}\left(1-\frac{t\cos\varphi}{\sqrt{t^2\cos^2\varphi+3}}\right)F_l(\rho)\cos(l\varphi)\mid_{\rho=\sqrt{t^2\cos^2\varphi+3}-t\cos\varphi}\mathrm{d}\varphi\\
&=T_{\abs{l}}\left(\frac{1}{t}\right)\int_{0}^{\alpha(t)}\tilde{F}_l(t\cos\varphi)\cos(l\varphi)\mathrm{d}\varphi,
\end{split}
\end{equation}
where $\tilde{f}$ is 
defined as
\begin{equation}
\label{tilF}
    \tilde{f}(x)=\left(1-\frac{|x|}{\sqrt{|x|+3}}\right)f\left((\sqrt{|x|^2+3}-|x|)\cdot\frac{x}{|x|}\right)
\end{equation}
and $\tilde{F}(\rho,\alpha)
= \tilde{f}(\rho\theta(\alpha))$ is the polar form of $\tilde{f}$.
Note that $\tilde{F}$ is  in $L^\infty_c(B')$ where $B'$ is the exterior
of the closed unit ball.

After making the substitution $\rho=s\cos\varphi$, we have
\begin{equation}\label{eq0}
\frac{\paren{\mathcal{T}f}_l(t)}{4r}=T_{\abs{l}}\left(\frac{1}{t}\right)\int_{1}^{t}\frac{\tilde{F}_l(\rho)T_{\abs{l}}\left(\frac{\rho}{t}\right)}{\sqrt{t^2-\rho^2}}\mathrm{d}\rho.
\end{equation}

We claim that the
function $g_l$ defined
by
\begin{equation}\label{equ4a}
g_l(t)=\int_{1}^{t}\frac{\tilde{F}_l(\rho)T_{\abs{l}}\left(\frac{\rho}{t}\right)}{\sqrt{t^2-\rho^2}}\mathrm{d}\rho
\end{equation}
is continuous on $[1,\infty)$.  To show this, one just writes
$g_l(t)-g_l(s)$ for $s<t$ as an integral on $[s,t]$ plus an integral
on $[1,s]$.  Because $\tilde{F}_l\in L^{\infty}([1,\infty))$, the
integral on $[s,t]$ clearly goes to zero as $s\to t$.  To show the
integral on $[1,s]$ goes to zero as $s\to t$, one makes the change of
variable $u=s-\rho$ and then uses Dominated Convergence on the
integrand to show it converges to zero, too (after assuming $s>t/2$).
In this case, the integrand is bounded near the endpoint that depends
on $s$.  The proof of continuity if $t<s$ uses similar ideas;
dominated convergence works on the integral on $[1,t]$ and the
integral on $[t,s]$ requires the change of variable.

Now, assume that $\paren{\mathcal{T}f}_l=0$. Since $g_l$ is continuous,
$g_l=0$ everywhere. So we have
\begin{equation}
\label{equ4}
\int_{1}^{t}\frac{\tilde{F}_l(\rho)T_{\abs{l}}\left(\frac{\rho}{t}\right)}{\sqrt{t^2-\rho^2}}\mathrm{d}\rho=0
\end{equation}
for all $t\in (1,\infty)$.  Then, equation \eqref{equ4a} is a generalized
Abel integral equation of the first kind and the right-hand side is absolutely
continuous The kernel is
\[\bparen{\frac{T_{\abs{l}}\paren{\frac{\rho}{t}}}{\sqrt{t+\rho}}}\frac{1}{\sqrt{t-\rho}}
,\] and the term in brackets is nonzero when $t=\rho$.  Using this
information and arguments in \cite{Tricomi, Yoshida} and stated in
\cite[Theorem B]{Q1983a}, one sees that $f_l=0$ and thus $\Tc$ is
invertible on domain $L^\infty(B)$.\end{proof}


\begin{remark}\label{rem:Cormack} The integral equation in
\eqref{eq0} provides a method to reconstruct the polar Fourier
coefficients of $f$ from the data. If one lets 
\[g_l(t) = \frac{\paren{\Tc f}_l(t)}
{4rT_{\abs{l}}\paren{\frac{1}{t}}},
\]then \eqref{eq0} becomes \eqref{equ4a}.  With a simple change of
variables in \eqref{equ4a}, $r=1/\rho$ and letting $p=1/t$ one reduces
the integral on the right-hand side of \eqref{equ4a} essentially to the
integral equation in \cite[equation (10)]{cormack} for the
$l^\text{th}$ polar Fourier coefficient a function that is the product
of a nonzero function and a composition of $f_l$ with a
diffeomorphism.

Cormack inverts his expression \cite[equation (10)]{cormack} by
another Abel type equation (see \cite[equations (17) and
(18)]{cormack}), and this would give the related function and hence,
$f$.  However, this inversion formula is numerically unstable because
it involves $T_l(p/z)$ where $p>z$ and $T_l(p/z)$ blows up like
$(p/z)^l$.  This is why Cormack developed a different reconstruction
method for X-ray CT using an SVD in \cite{cormack1964}.
\end{remark}

So far we have shown that the problem of reconstructing a density
$f\in L^{\infty}(B))$ from $\mathcal{T}f$ is uniquely solvable, and
provided explicit expressions for the expected artefacts in the
reconstruction. We next go on to demonstrate our theory through
discrete simulations.


\section{Reconstruction algorithm and results}\label{res}

Here we present reconstruction algorithms for the reconstruction of two dimensional densities from toric section integral data and demonstrate the artefacts described by the theory in section \ref{tsec1}. 

We take a discrete (algebraic) approach to reconstruction. That is we discretize the operator $\mathcal{T}$ on a pixel grid (see figure \ref{fig1.1}) and find
\begin{equation}
\label{recmth}
\argmin_{\vv}\|A\vv-b\|^2_2+\lambda^2\mathcal{G}(\vv),
\end{equation}
where $A$ is the discrete form of $\mathcal{T}$ (each row of $A$ is the vectorized form of a binary image as shown in figure \ref{fig1.1}) and $\mathcal{G}(x)$ is a regularization penalty (e.g. $\mathcal{G}(\vv)=\|\vv\|^2_2$ (Tikhonov) or $\mathcal{G}(\vv)=\sum_i|\vv_i-\vv_{i-1}|$ (TV)), with regularization parameter $\lambda$. Here $\vv$ represents the vectorized form of the density image (which is to be reconstructed) and $b$ (our data) represents the Compton scattered intensity.

To simulate noisy data we take a vectorized density image $x$ (such as those presented in figure \ref{fig4}) we add a Gaussian random noise
\begin{equation}
\label{noise}
\vb=A\vv+\epsilon\times\frac{\vg\|A\vv\|_2}{\sqrt{n}},
\end{equation}
where $\vg$ is a pseudo-random vector of samples drawn from a standard normal distribution and $n$ is the number of entries in $\vb$. Here $\epsilon$ denotes the noise level in the sense that
$$\frac{\|\vb-A\vv\|_2}{\|A\vv\|_2}\approx\epsilon$$
for $n$ large enough. It is noted that simulating data as
  in \eqref{noise} can often lead to optimistic results (due to the inverse crime). In appendix \ref{app2} we present additional reconstructions of a``multiple ring" phantom using analytically generated toric integral data, to avoid the inverse crime. The ring phantom $f$ is such that a closed form for $\Tc f$ is possible. For the more general phantoms considered later in this section, we have not found such a closed form. Hence in the main text, we choose to simulate the data as in \eqref{noise}. We shall see later (in figure \ref{drec}) that the artefacts predicted by our microlocal theory are present using \eqref{noise} for data simulation, so such a data generation is sufficient to verify our theoretical results.

Throughout the simulations presented here we simulate toric section integral data for rotation angles $\alpha\in\left\{\frac{j\pi}{180} : 1\leq j\leq 360\right\}$ and for circle radii $r\in\left\{\frac{j^2+200^2}{2j} : 1\leq j\leq 199\right\}$, where the pixel grid size is 200--200. So $n=360\times 199=71640$ and $A$ has $200^2$ columns.
\begin{figure}[!h]
\begin{subfigure}{0.47\textwidth}
\includegraphics[width=0.9\linewidth, height=0.9\linewidth]{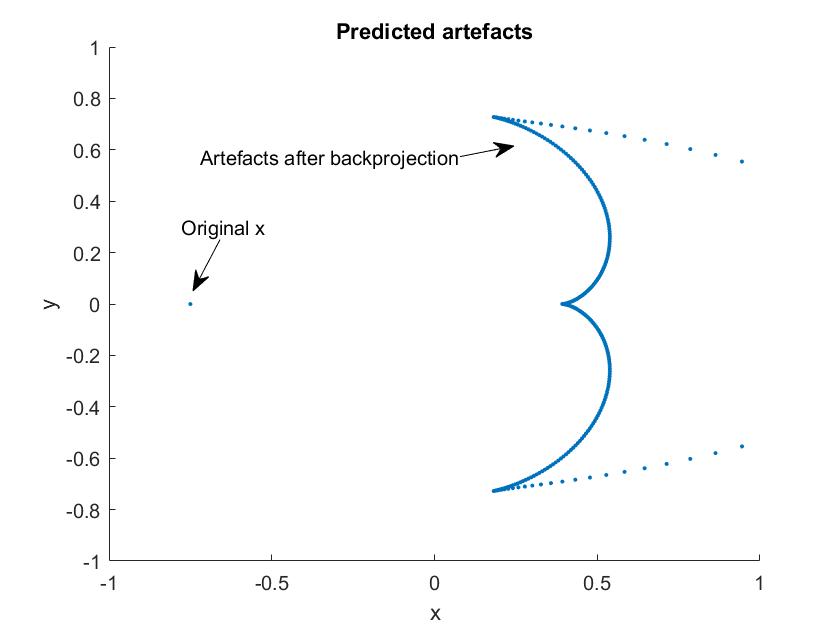} 
\end{subfigure}
\begin{subfigure}{0.47\textwidth}
\includegraphics[width=1.0\linewidth, height=0.9\linewidth]{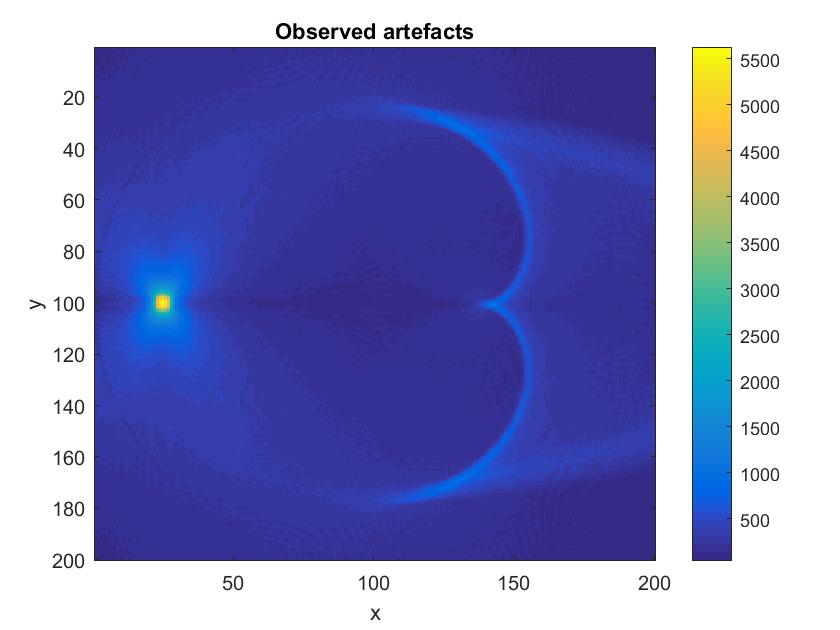}
\end{subfigure}
\caption{Predicted and observed artefacts from reconstructing a delta function far from the origin by backprojection.}
\label{fig8}
\end{figure}
\begin{figure}[!h]
\begin{subfigure}{0.47\textwidth}
\includegraphics[width=0.9\linewidth, height=0.9\linewidth]{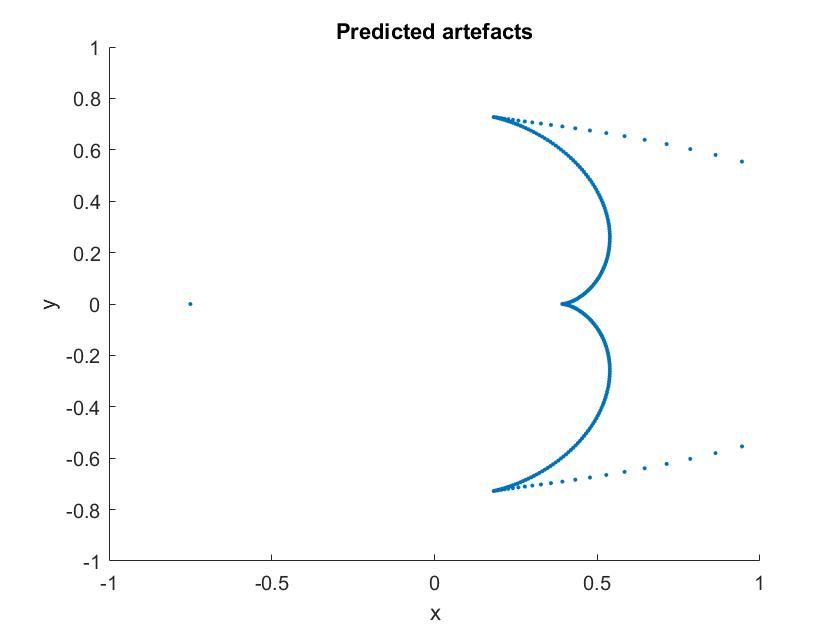} 
\end{subfigure}
\begin{subfigure}{0.47\textwidth}
\includegraphics[width=1.0\linewidth, height=0.9\linewidth]{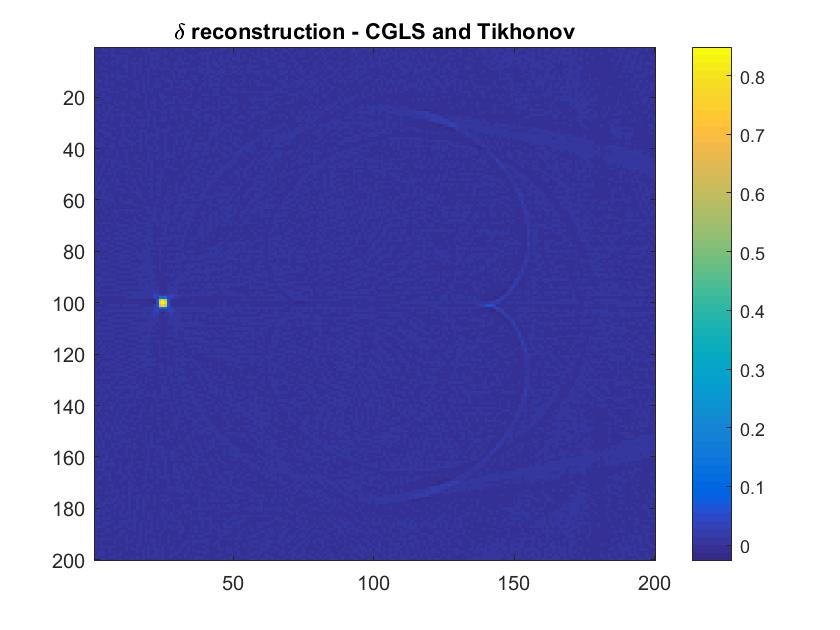}
\end{subfigure}
\begin{subfigure}{0.47\textwidth}
\includegraphics[width=1.0\linewidth, height=0.9\linewidth]{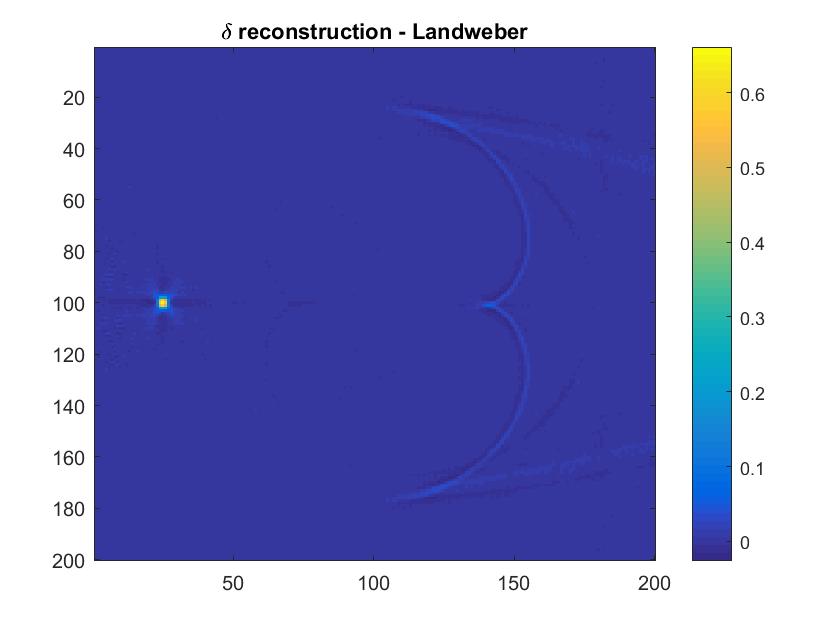} 
\end{subfigure}
\begin{subfigure}{0.47\textwidth}
\includegraphics[width=1.0\linewidth, height=0.9\linewidth]{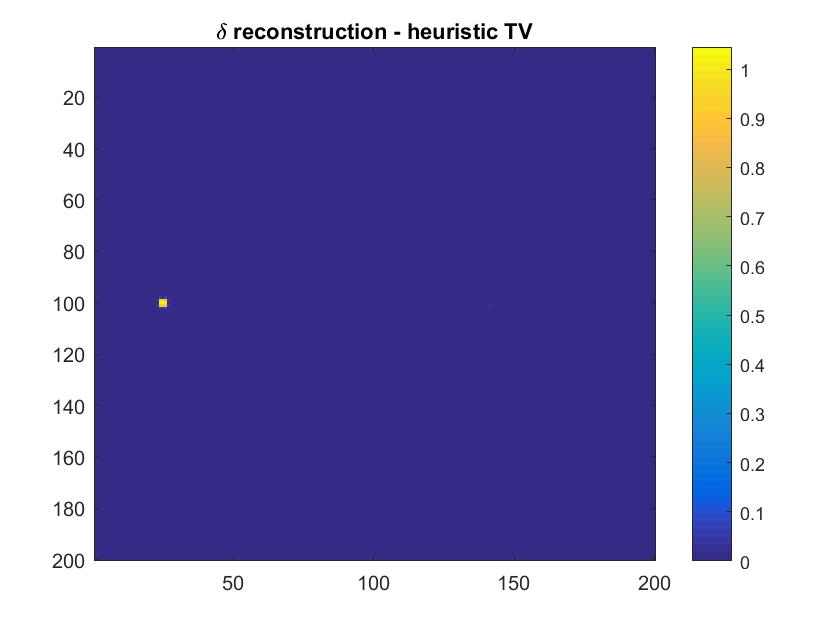}
\end{subfigure}
\caption{Reconstructions of a delta function with $5\%$ added noise. Top left --  Predicted artefacts. Top right -- CGLS and Tikhonov.  Bottom left -- Landweber. Bottom right -- heuristic TV.}
\label{drec}
\end{figure}

To simulate the artefacts implied by the theory presented in section \ref{tsec1}, we consider the reconstruction of a delta function by (unfiltered) backprojection. That is by an application of the normal operator $\mathcal{T}^*\mathcal{T}\delta$, where $\delta$ has its support in the unit ball. To calculate the artefacts induced by $\Lambda_1=\tCc_1^t\circ\Cc_2$ and $\Lambda_2=\tCc_2^t\circ\Cc_1$ (as in Theorem \ref{thm:main}) when $f=\delta$ (so here $f$ is non zero only at a single point and its wavefront set lies in all directions), let us consider a point $\vx=|\vx|(-1,0)$ on the $x$ axis. Then equation \eqref{art1} becomes
\begin{equation}
    \vy=\paren{1+\frac{2}{s}\sin\alpha\cos\alpha,\frac{2}{s}\sin^2\alpha}^T
\end{equation}
up to scaling. Similarly for $\vy=|\vy|(-1,0)$ equation \eqref{art2} becomes
\begin{equation}
    \vx=\paren{1-\frac{2}{s}\sin\alpha\cos\alpha,-\frac{2}{s}\sin^2\alpha}^T,
\end{equation}
again up to scaling. Let us define $\psi_1 : [0,\pi]\to\text{sg}(\mathbb{R}^2)$ and $\psi_2 : [-\pi,0]\to\text{sg}(\mathbb{R}^2)$ as
\begin{equation}
\label{art4}
    \psi_1(\alpha)=\left\{\nu\paren{1-\frac{2}{s}\sin\alpha\cos\alpha,-\frac{2}{s}\sin^2\alpha}:\nu\in\mathbb{R}\right\}
    \cap C_1\cap\left\{\vx\cdot\theta_{\alpha}<0\right\}.
\end{equation}
and
\begin{equation}
\label{art3}
    \psi_2(\alpha)=\left\{\nu\paren{1+\frac{2}{s}\sin\alpha\cos\alpha,\frac{2}{s}\sin^2\alpha} :\nu\in\mathbb{R}\right\}\cap C_2\cap\left\{\vx\cdot\theta_{\alpha}>0\right\}
\end{equation}
where $\text{sg}(\mathbb{R}^2)$ denotes the set of singleton subsets of $\mathbb{R}^2$. Also
\begin{equation}
    s=\left|\frac{3-|\vx|^2+2(\vx\cdot\theta)}{2(\vx\cdot\theta_{\alpha})}\right|,
\end{equation}
to get $s$ in terms of $\vx$ and a rotation $\alpha$. Then
$\psi_1([0,\pi])$ and $\psi_2([-\pi,0])$ are the set of artefacts in the plane associated to $\Lambda_1$ and $\Lambda_2$ respectively. Note that we need only consider the domain $[0,\pi]$ for $\psi_1$ as the circle $C_1$ does not intersect $\vx=|\vx|(-1,0)$ for any $\alpha\in (0,\pi)$, and conversely for $\psi_2$. It is clear that $\psi_1([0,\pi])=P\psi_2([-\pi,0])$, where $P$ denotes a reflection in the line $\{t\vx : t\in\mathbb{R}\}$ (or the $x$ axis in this case). Hence the artefacts associated to $\Lambda_1$ are those associated to $\Lambda_2$ but reflected in the line $\{t\vx : t\in\mathbb{R}\}$, for a given $\vx\in\mathbb{R}^2$, when $f$ has singularities at $\vx$ in all directions $\xi$. We can use equations \eqref{art3} and \eqref{art4} to draw curves in the plane where we expect there to be image artefacts. To simulate $\delta$ discretely we assign a value of 1 to nine neighbouring pixels in the unit cube (discretized as a 200--200 grid) and set all other pixel values to zero. Let our discrete delta function be denoted by $\vv_{\delta}$. Then we approximate $\mathcal{T}^*\mathcal{T}\delta\approx A^TA\vv_{\delta}$.
\begin{figure}[!h]
\begin{subfigure}{0.47\textwidth}
\includegraphics[width=0.9\linewidth, height=0.9\linewidth]{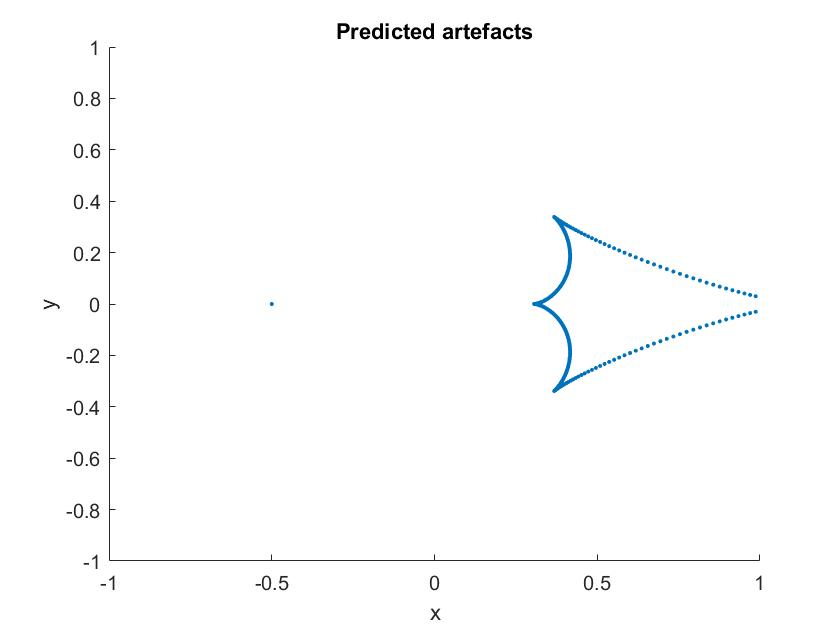} 
\end{subfigure}
\begin{subfigure}{0.47\textwidth}
\includegraphics[width=1.0\linewidth, height=0.9\linewidth]{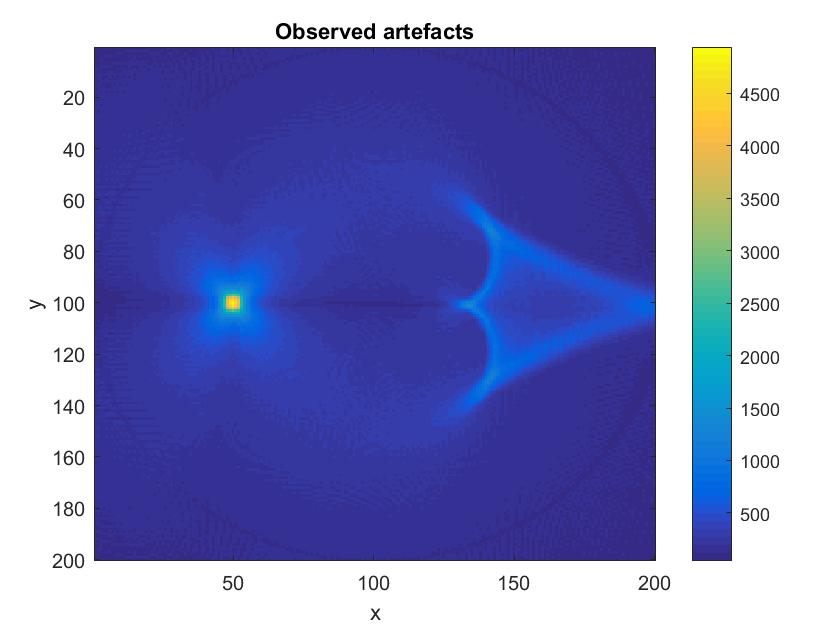}
\end{subfigure}
\caption{Predicted and observed artefacts from reconstructing a delta function closer to the origin by backprojection.}
\label{fig9}
\end{figure}
See figures \ref{fig8} and \ref{fig9}, where we have shown side by
side comparisons of the artefacts predicted by equations \eqref{art3} and \eqref{art4}
and the artefacts observed in a reconstruction by backprojection. See
also figures \ref{fig8.1} and \ref{fig9.1} for more simulated artefact
curves. Note that the blue dots in the left hand figures are the outputs of $\psi_1$ for $\alpha\in\left\{\frac{j\pi}{180} : 1\leq j\leq 180\right\}$ and $\psi_2$ for $\alpha\in\left\{-\frac{j\pi}{180} : 1\leq j\leq 180\right\}$. The observed artefacts are as predicted by the theory and the
images in the left and right hand sides of each figure superimpose
exactly. We notice a cardioid curve artefact in the reconstruction
which becomes a full cardioid when the delta function lies
approximately on the unit circle. 

To test our reconstruction techniques, we consider the test phantoms displayed in figure \ref{fig4}, one simple and one complex. 
\begin{figure}[!h]
\begin{subfigure}{0.47\textwidth}
\includegraphics[width=1.0\linewidth, height=0.9\linewidth]{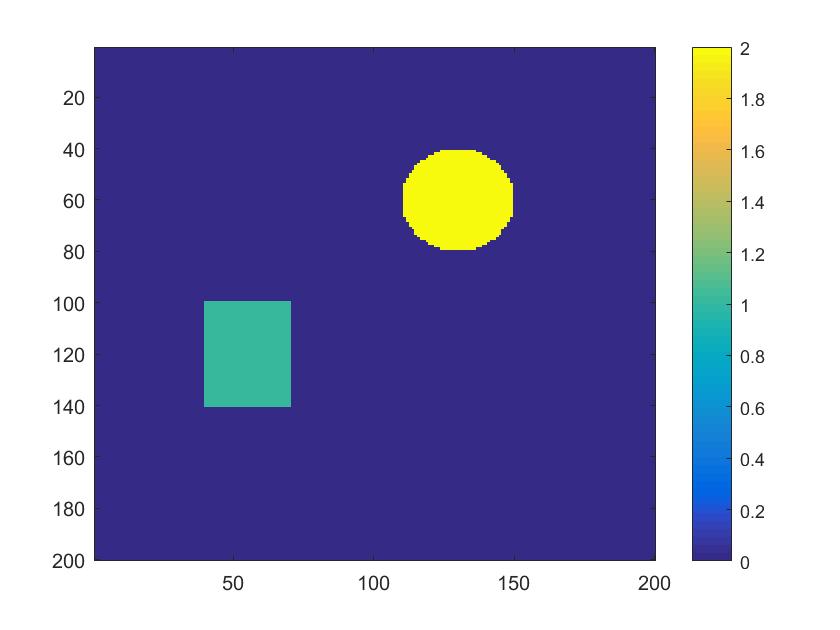} 
\end{subfigure}
\begin{subfigure}{0.47\textwidth}
\includegraphics[width=1.0\linewidth, height=0.9\linewidth]{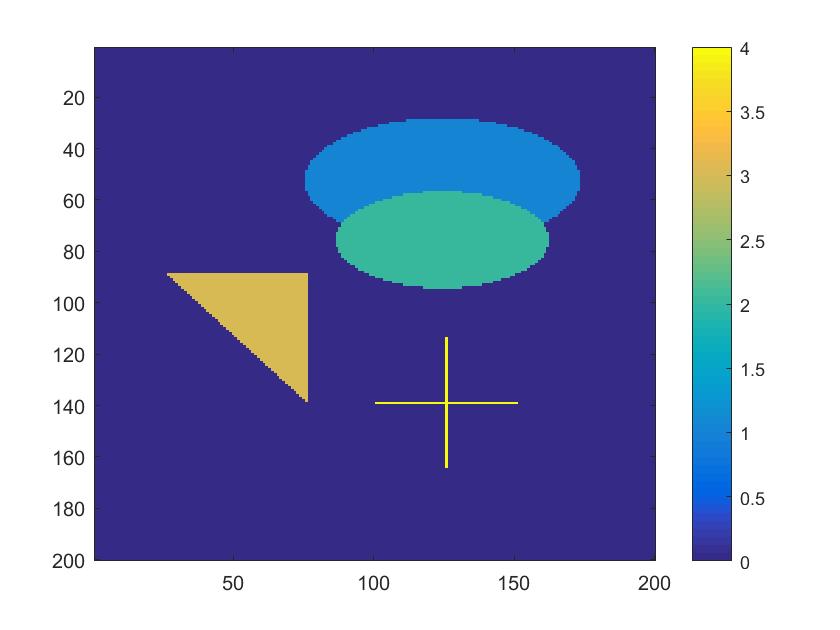}
\end{subfigure}
\caption{Simple (left) and complex (right) phantoms.}
\label{fig4}
\end{figure}
The simple phantom consists of a disc with value 2 and a square with value 1. The complex phantom consists of simulated objects of varying density, shape and size with overlapping ellipsoids, and is commonly used to test reconstruction techniques in tomography \cite{hansen}. See figures \ref{fig6}, \ref{fig7}, \ref{fig5}, \ref{fig5.1} for reconstructions of the two test phantoms using the Landweber method and a Conjugate Gradient Least Squares (CGLS) iterative solver \cite{hansen} with Tikhonov regularization (varying the regularization parameter $\lambda$ manually). In the absence of noise ($\epsilon=0$) there are significant artefacts in the reconstruction using a Landweber approach. CGLS performs well however on both test phantoms. In the presence of added noise (we consider noise levels of $1\%$ ($\epsilon=0.01$) and $5\%$ ($\epsilon=0.05$)) there are severe artefacts in the reconstruction using a CGLS with Tikhonov approach (see figures \ref{fig6} and \ref{fig7}), particularly with a higher noise level of $5\%$.

\begin{figure}[!h]
\begin{subfigure}{0.47\textwidth}
\includegraphics[width=1.0\linewidth, height=0.9\linewidth]{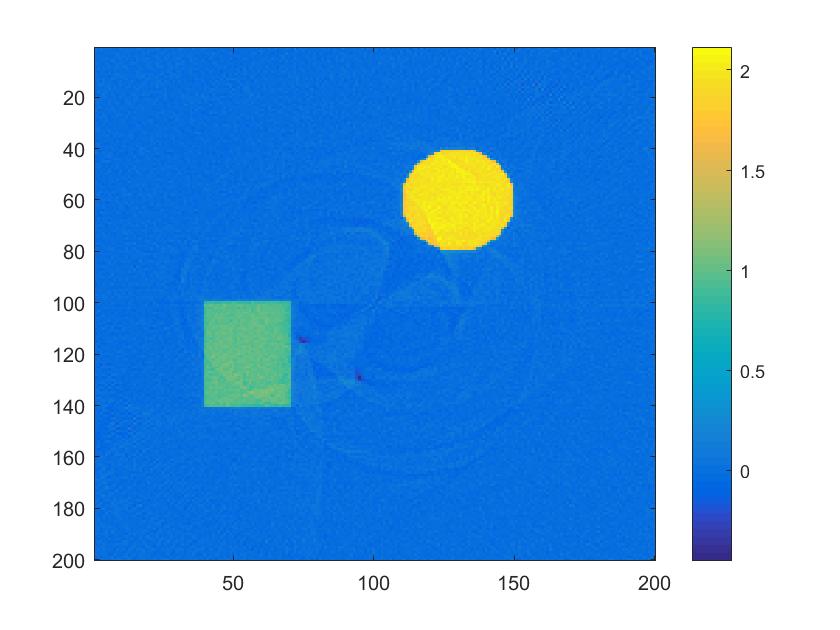} 
\end{subfigure}
\begin{subfigure}{0.47\textwidth}
\includegraphics[width=1.0\linewidth, height=0.9\linewidth]{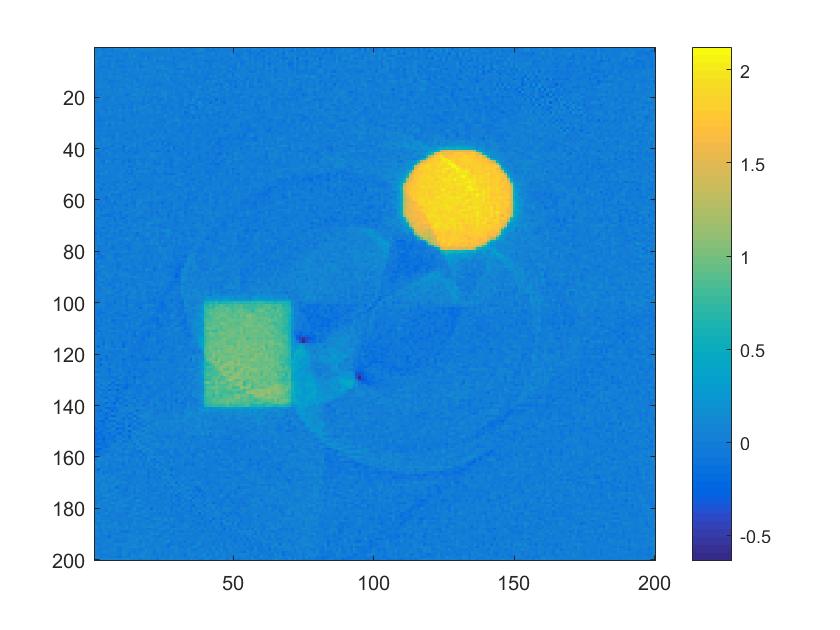}
\end{subfigure}
\caption{Simple phantom reconstruction using CGLS and Tikhonov as a regularizer, with noise levels of $1\%$ (left) and $5\%$ (right).}
\label{fig6}
\end{figure}
\begin{figure}[!h]
\begin{subfigure}{0.47\textwidth}
\includegraphics[width=1.0\linewidth, height=0.9\linewidth]{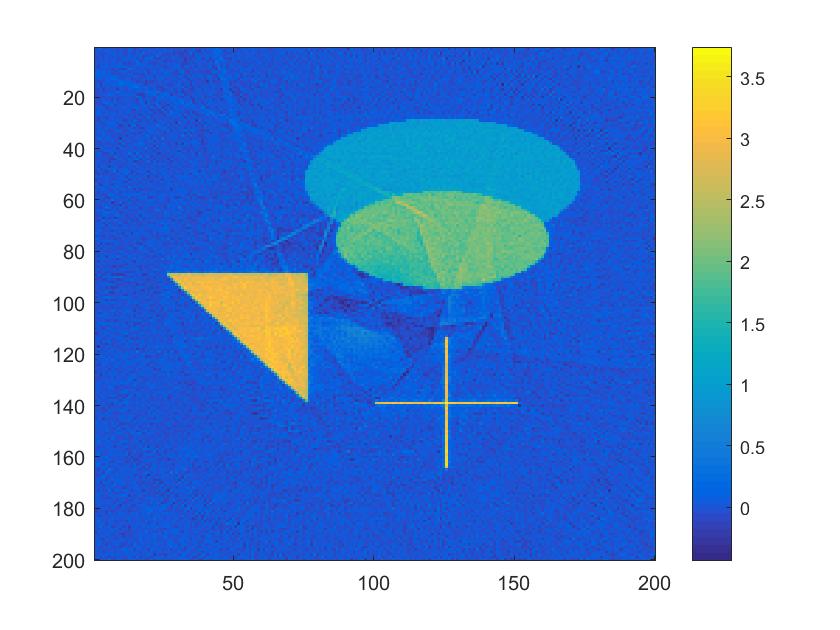} 
\end{subfigure}
\begin{subfigure}{0.47\textwidth}
\includegraphics[width=1.0\linewidth, height=0.9\linewidth]{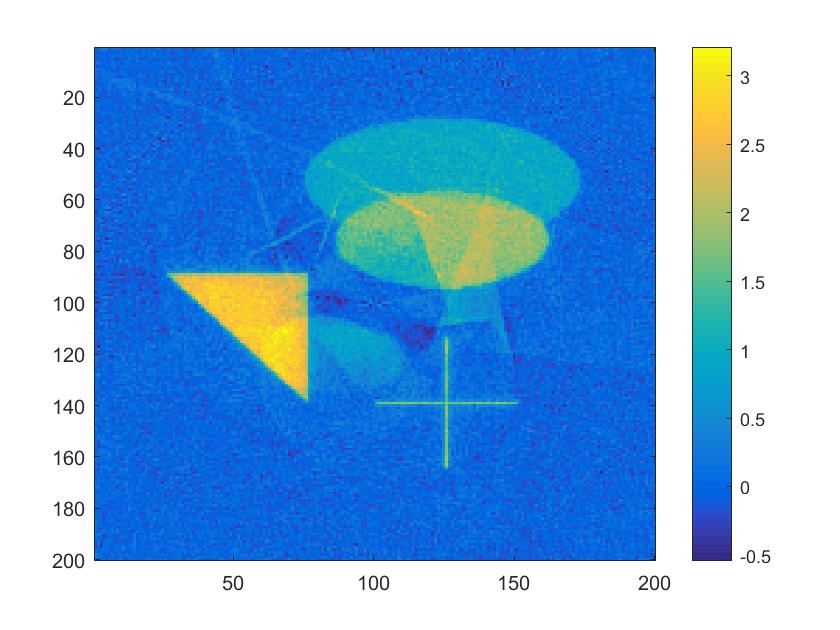}
\end{subfigure}
\caption{Complex phantom reconstruction using CGLS and Tikhonov as a regularizer, with noise levels of $1\%$ (left) and $5\%$ (right).}
\label{fig7}
\end{figure}
To combat the image artefacts we found that the use of an iterative approach with heuristic TV regularization (as described in \cite{IRtools}) was effective. Specifically we apply the method ``IRhtv" of \cite{IRtools} with added non--negativity constraints to the optimizer (as we know a--priori that a density is non--negative), and choose the regularization parameter $\lambda$ manually. For more details on the IRhtv method see \cite{gazzola2014generalized}. See figures \ref{fig6.1} and \ref{fig7.1}. 
\begin{figure}[!h]
\begin{subfigure}{0.47\textwidth}
\includegraphics[width=1.0\linewidth, height=0.9\linewidth]{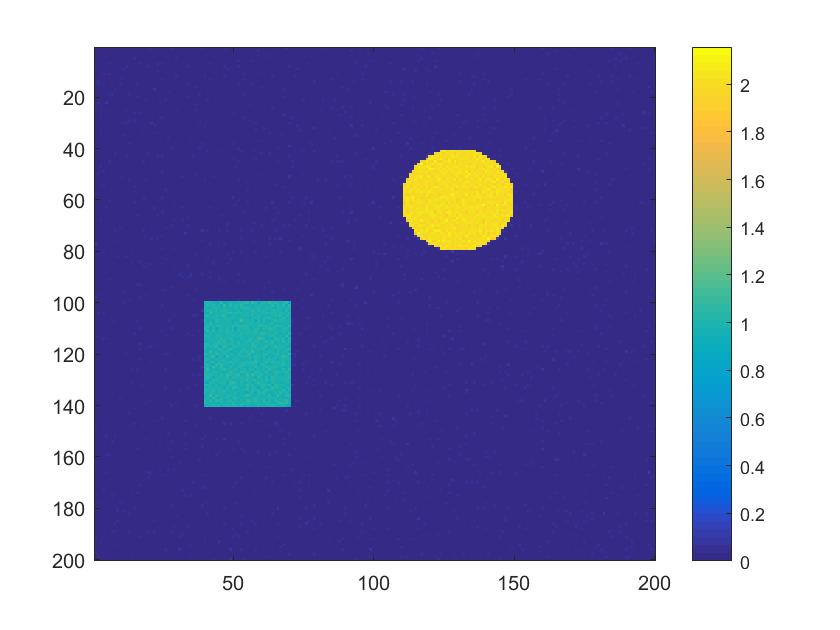} 
\end{subfigure}
\begin{subfigure}{0.47\textwidth}
\includegraphics[width=1.0\linewidth, height=0.9\linewidth]{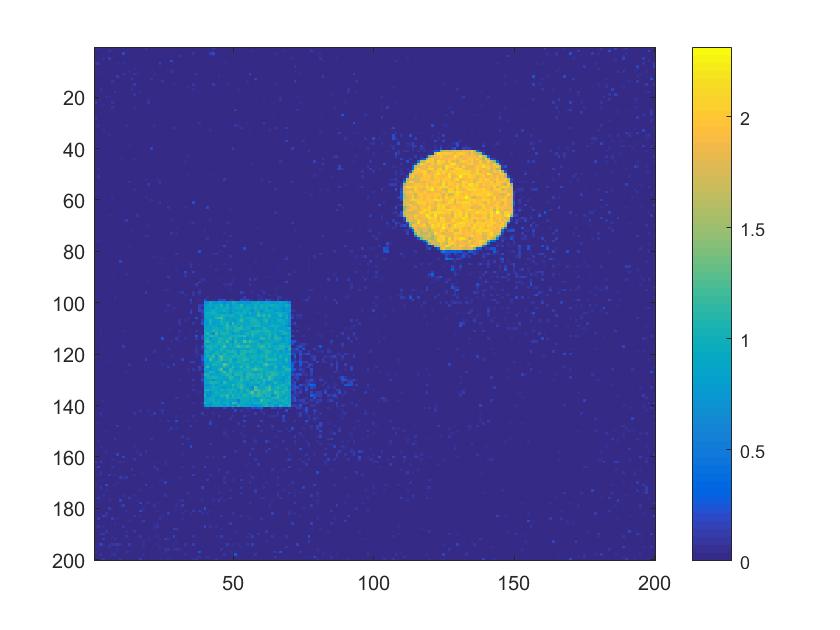}
\end{subfigure}
\caption{Simple phantom reconstruction using a heuristic TV regularizer, with noise levels of $1\%$ (left) and $5\%$ (right).}
\label{fig6.1}
\end{figure}
\begin{figure}[!h]
\begin{subfigure}{0.47\textwidth}
\includegraphics[width=1.0\linewidth, height=0.9\linewidth]{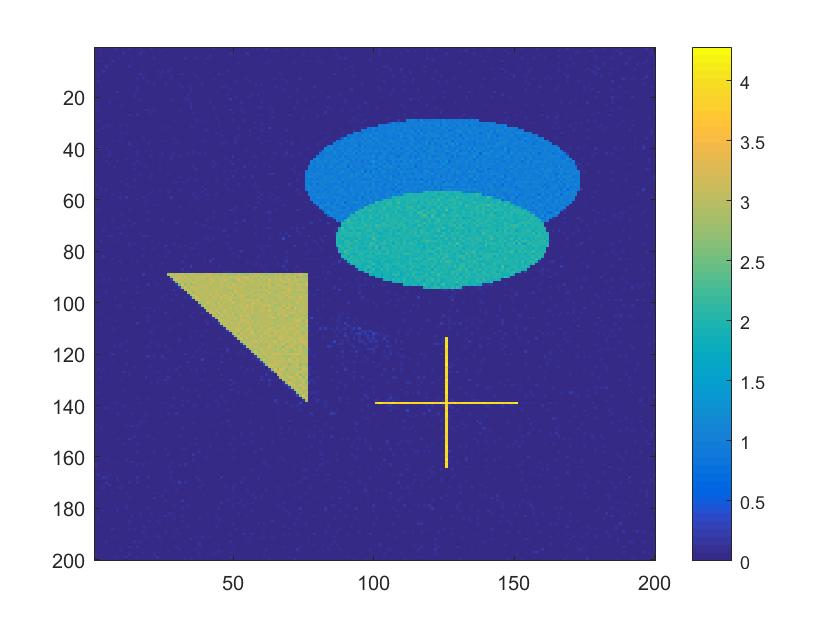} 
\end{subfigure}
\begin{subfigure}{0.47\textwidth}
\includegraphics[width=1.0\linewidth, height=0.9\linewidth]{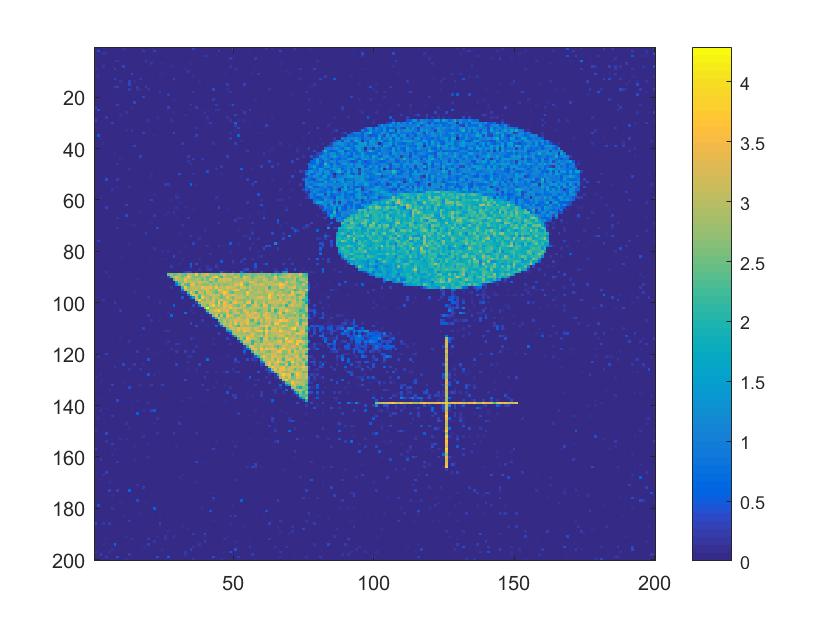}
\end{subfigure}
\caption{Complex phantom reconstruction using a heuristic TV regularizer, with noise levels of $1\%$ (left) and $5\%$ (right).}
\label{fig7.1}
\end{figure}
For a noise level of $1\%$ the artefacts are almost completely removed from the reconstructions (for both the simple and complex phantom) and the image quality is high overall. For a higher noise level of $5\%$ we see a significant reduction in the artefacts and the reconstruction is satisfactory in both cases with a low level of distortion in the image (although there is a higher distortion in the complex phantom reconstruction).

The predicted artefacts of figures \ref{fig8} and \ref{fig9} are also observed in a discrete reconstruction. See figure \ref{drec}, where we have presented reconstructions of a delta function using the three iterative methods considered in this paper, namely CGLS with Tikhonov, a Landweber iteration and the solvers of \cite{IRtools} with heuristic TV. The artefacts of figure \ref{fig8} can be observed faintly in the reconstruction using CGLS, and are most pronounced in the Landweber iteration. The heuristic TV approach gives the best performance (as before), although the reconstruction quality is more comparable among the three methods considered for a simple phantom such as a delta function.

For the application considered in this paper, namely threat detection in airport baggage screening, the removal of image artefacts and an accurate quantitative density estimation are crucial to maintain a satisfactory false positive rate. We will now further compare our results using CGLS with Tikhonov and the iterative solver of \cite{IRtools}, in terms of the false positive rate we can expect using both methods. Looking at the reconstructions using both methods qualitatively. In figure \ref{fig7} (using CGLS with Tikhonov), the image artefacts visually mask the four shapes which make up the original density. This may lead to threat materials or objects being misidentified (false negative errors). In addition, the artefacts introduce new ``fake" densities (e.g. streaks in the top left of the image) to the original, which may be wrongly interpreted as a potential threat by security personnel (a false positive error). In figure \ref{fig7.1} (using the iterative solver of \cite{IRtools}), with only a mild distortion in the image, we are less prone to such mistakes. 

For a brief quantitative analysis, let the ``cross" shaped object (with relative density 4) represent a detonator element and let the ``triangular" density (with relative density 3) represent a small plastic explosive. Then the presence of artefacts can introduce large errors in the density estimation. For example, let us consider the left hand image in figure \ref{fig7}. if we take the average pixel value of the reconstructed explosive and detonator, then the relative errors are
\begin{equation}
\text{errT}=100\times\frac{|\text{avgT}-3|}{3}=9.31\%,\ \ \ \ \text{errC}=100\times\frac{|\text{avgC}-4|}{4}=43.9\%,
\end{equation}
where $\text{avgT}=2.72$ and $\text{avgC}=2.25$ are the average pixel values for the reconstructed plastic explosive and detonator element respectively. Let us say we were using a look up table approach to threat detection (which is a common approach). That is we look for densities (of a large enough size) in a pre--specified set of values and flag these as a potential threat. In threat detection, we cannot allow any false negatives, so if the above error rates were as expected the space of potential threats (the set of suspicious density values) would have to be increased (to allow for errors up to $44\%$) in order to compensate and identify the explosive, thus increasing the false positive rate. 

If we now consider the same error rates for the left hand image in figure \ref{fig7.1}, then
\begin{equation}
\text{errT}=100\times\frac{|\text{avgT}-3|}{3}=0.27\%,\ \ \ \ \text{errC}=100\times\frac{|\text{avgC}-4|}{4}=2.00\%,
\end{equation}
where in this case $\text{avgT}=2.99$ and $\text{avgC}=3.92$. With such a reduction in the error rate, we can safely reduce the space of potential threats (now only allowing for errors less than $2\%$) in our look up table and hence reduce the expected false positive rate.

\section{Conclusion}
Here we have introduced a new toric section transform $\mathcal{T}$ which
describes a two dimensional Compton tomography problem in airport baggage
screening. A novel microlocal analysis of $\mathcal{T}$ was presented
whereby the reconstruction artefacts were explained through an analysis of
the canonical relation. This was carried out by an analysis of two circle
transforms $\mathcal{T}_1$ and $\mathcal{T}_2$, whose canonical relations
($\mathcal{C}_1$ and $\mathcal{C}_2$) were shown to satisfy the Bolker
Assumption when considered separately. When we considered their disjoint
union ($\mathcal{C}=\mathcal{C}_1\cup\mathcal{C}_2$), which describes the
canonical relation of $\mathcal{T}$, this was shown to be 2--1. We gave
explicit expressions for the image artefacts implied by the 2--1 nature of
$\mathcal{C}$ in section \ref{tsec1}.

The injectivity of $\mathcal{T}$ was proven on the set of $L^\infty$
functions $f$ with compact support in $B$. Here we used the
parameterization of circular arcs given by Nguyen and Truong in
\cite{NT} to decompose $\mathcal{T}f$ in terms of orthogonal special
functions (exploiting the rotational symmetry of $\mathcal{T}f$), and
then applied similar ideas to those of Cormack \cite{cormack} to prove
injectivity. 

In section \ref{res} we presented a practical reconstruction algorithm
for the reconstruction of densities from toric section integral data
using an algebraic approach. We proposed to discretize the linear
operator $\mathcal{T}$ on pixel grids (with the discrete form of
$\mathcal{T}$ stored as a sparse matrix) and to solve the
corresponding set of linear equations by minimizing the least squares
error with regularization. To do this we applied the iterative
techniques included in the package \cite{IRtools} and provided
simulated reconstructions of two test phantoms (one simple and one
complex) with varying levels of added pseudo-random noise. Here we
demonstrated the artefacts explained by our microlocal analysis
through a discrete application of the normal operator of $\mathcal{T}$
to a delta function, and showed (with a side by side comparison) that
the artefacts in the reconstruction were exactly as predicted by our
theory. We also showed that we could combat the artefacts in the
reconstruction effectively using an iterative solver with a
heuristic total variation penalty (using the code included in \cite{IRtools} for
solving large scale image reconstruction problems), and explained how
the improved artefact reduction implies a reduction in the false
positive rate in the proposed application in airport baggage
screening.

For further work we aim to consider more general acquisition
geometries for the reconstruction of densities from toric section
integral data in Compton scattering tomography. Here we have
considered the particular three dimensional set of toric sections
which describe the loci of scatterers for an idealised geometry for an
airport baggage scanner. We wonder if the 2--1 nature of the canonical
relation (or reflection artefacts) will be present for other toric
section transforms and we aim to say something more concrete about
this. For example, are reflection artefacts present or is the
canonical relation 2--1 for any toric section transform?

\begin{figure}[!h]
\begin{subfigure}{0.47\textwidth}
\includegraphics[width=1.0\linewidth, height=0.9\linewidth]{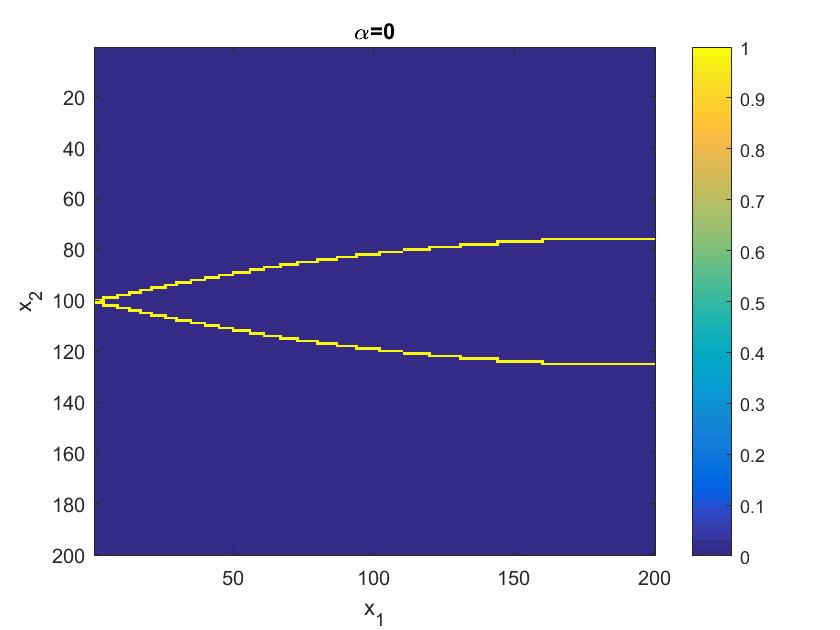}
\end{subfigure}
\begin{subfigure}{0.47\textwidth}
\includegraphics[width=1.0\linewidth, height=0.9\linewidth]{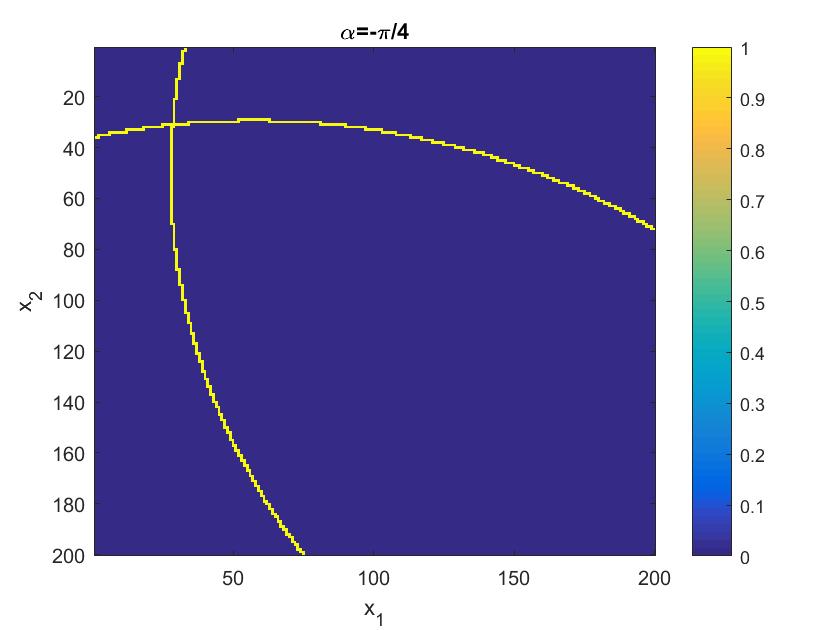}
\end{subfigure}
\begin{subfigure}{0.47\textwidth}
\includegraphics[width=1.0\linewidth, height=0.9\linewidth]{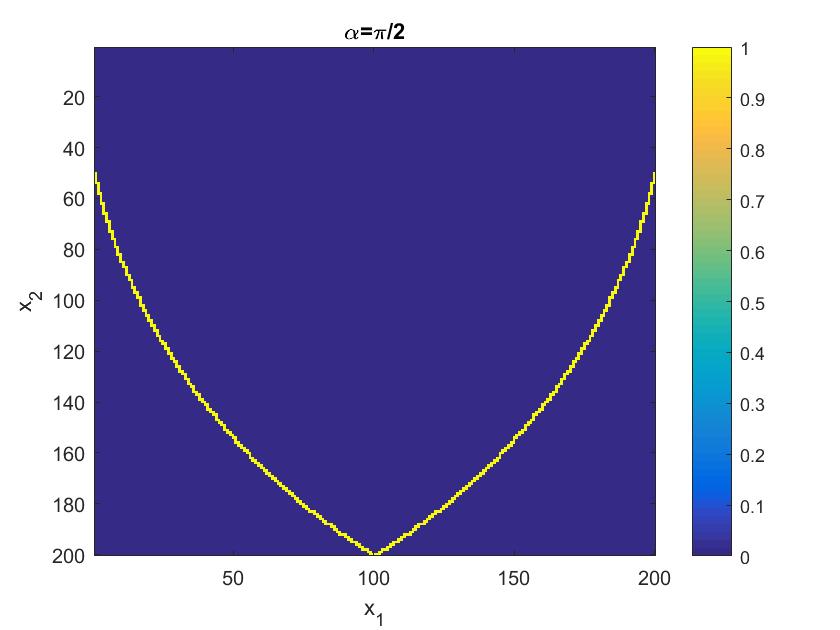} 
\end{subfigure}
\begin{subfigure}{0.47\textwidth}
\includegraphics[width=1.0\linewidth, height=0.9\linewidth]{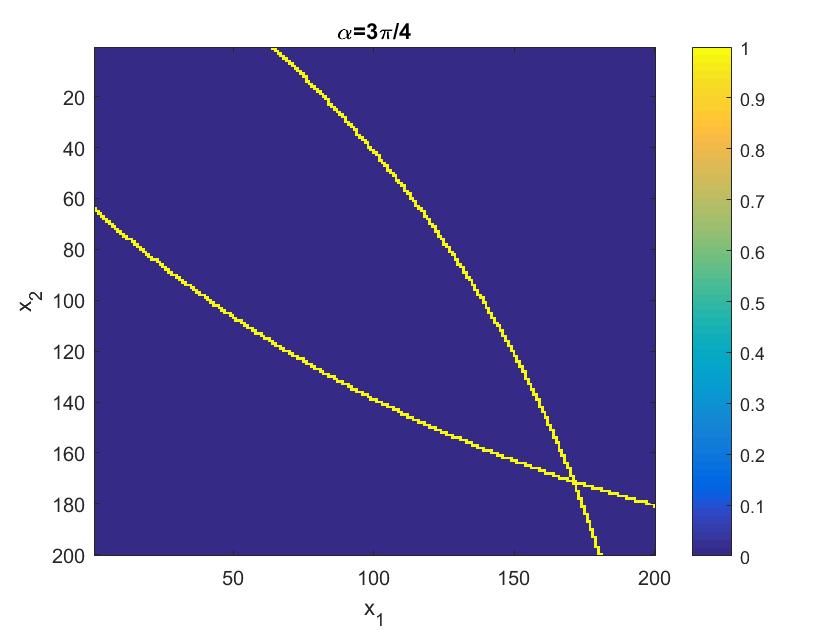}
\end{subfigure}
\caption{Discretized toric section integrals for varying rotation angles $\alpha$ and radii $r$ are presented as images. The images are binary (the pixel value is 1 if it intersected by a toric section and 0 otherwise).}
\label{fig1.1}
\end{figure}
\begin{figure}[!h]
\begin{subfigure}{0.47\textwidth}
\includegraphics[width=0.9\linewidth, height=7cm]{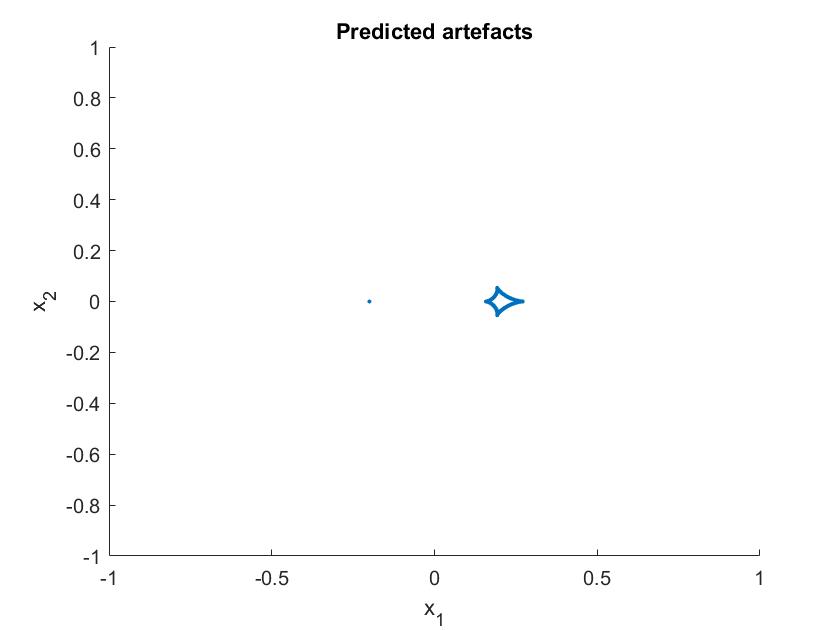} 
\end{subfigure}
\begin{subfigure}{0.47\textwidth}
\includegraphics[width=1.0\linewidth, height=7cm]{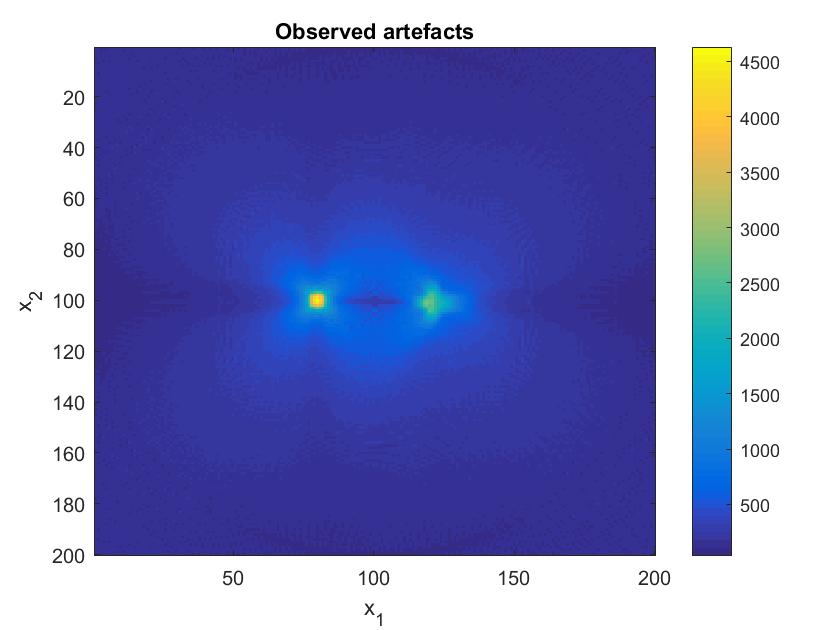}
\end{subfigure}
\caption{Predicted and observed artefacts from reconstructing a delta function close to the origin by backprojection.}
\label{fig8.1}
\end{figure}
\begin{figure}[!h]
\begin{subfigure}{0.47\textwidth}
\includegraphics[width=0.9\linewidth, height=7cm]{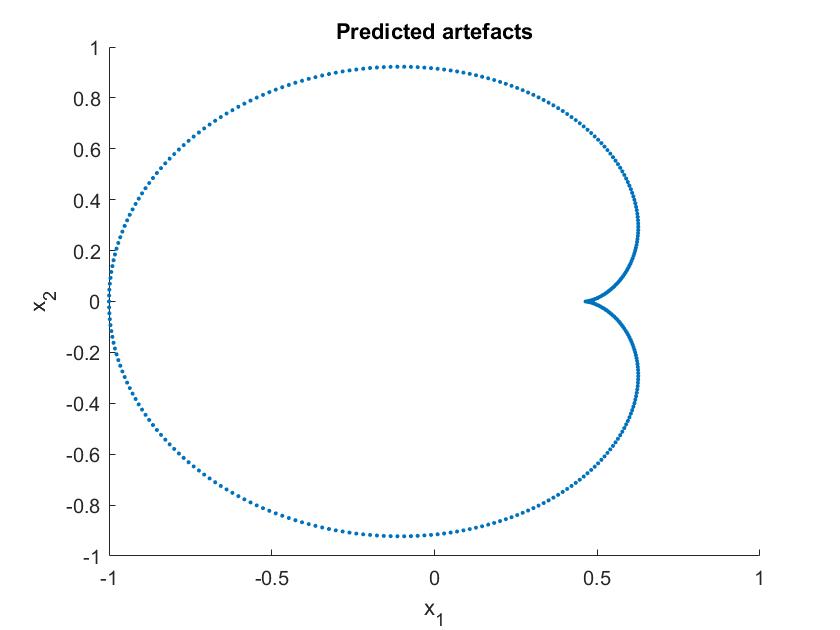} 
\end{subfigure}
\begin{subfigure}{0.47\textwidth}
\includegraphics[width=1.0\linewidth, height=7cm]{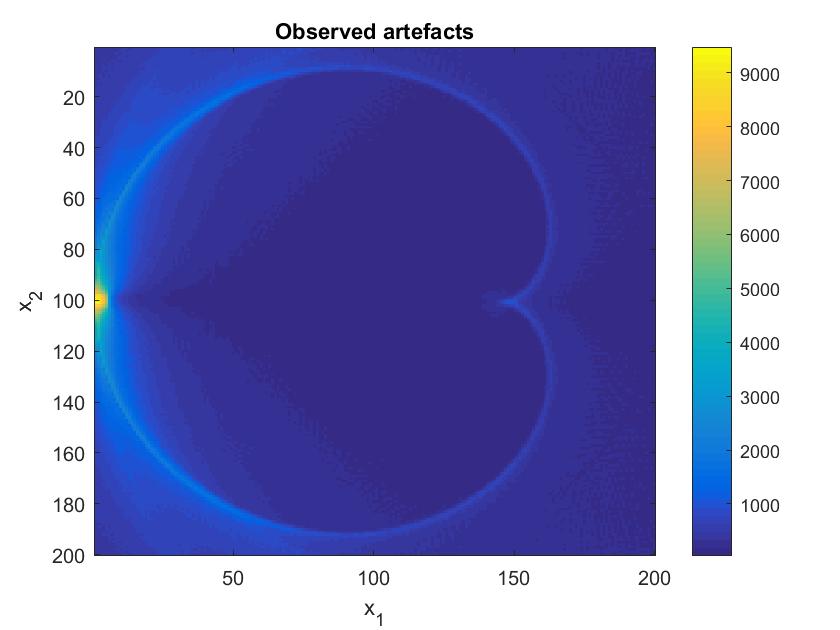}
\end{subfigure}
\caption{Predicted and observed artefacts from reconstructing a delta function on the boundary of the unit ball by backprojection. The artefacts are described by a cardioid.}
\label{fig9.1}
\end{figure}
\begin{figure}[!h]
\begin{subfigure}{0.47\textwidth}
\includegraphics[width=1.0\linewidth, height=0.9\linewidth]{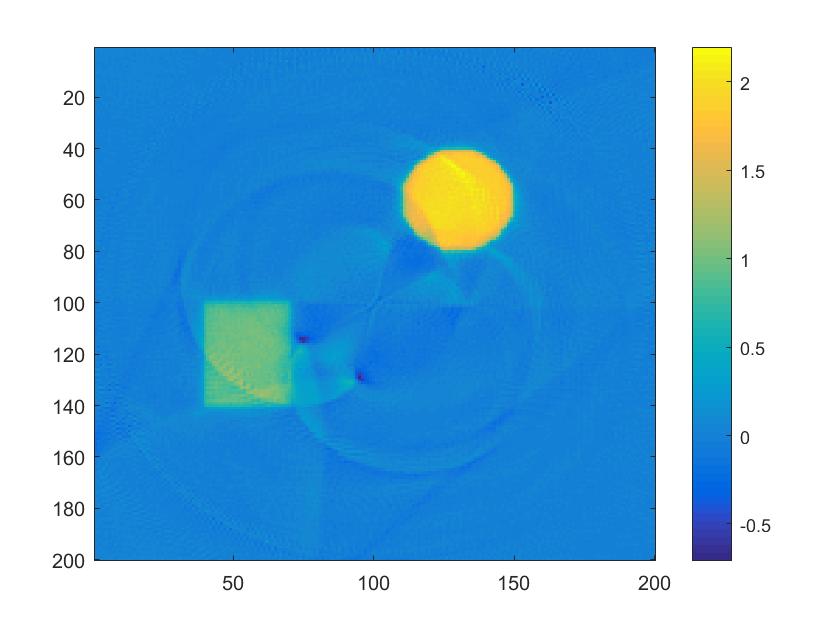} 
\end{subfigure}
\begin{subfigure}{0.47\textwidth}
\includegraphics[width=1.0\linewidth, height=0.9\linewidth]{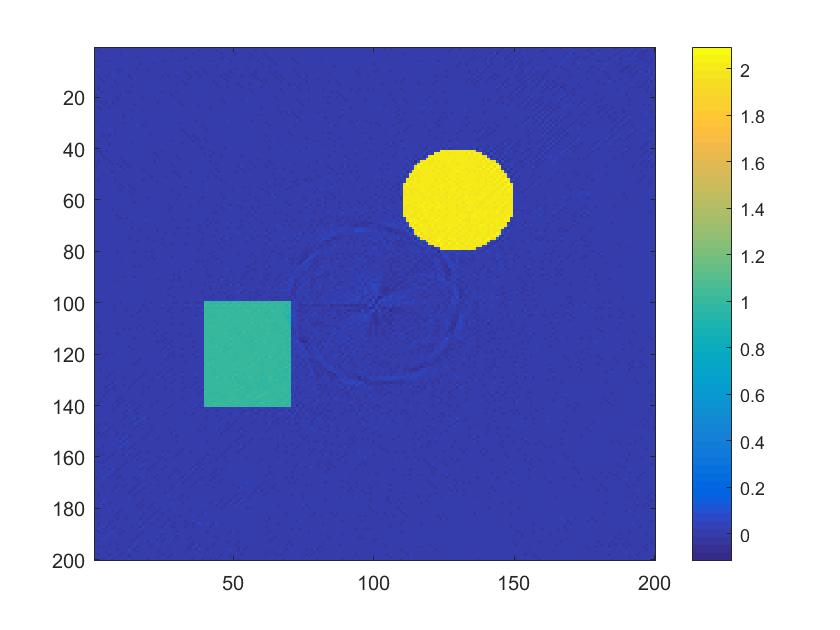}
\end{subfigure}
\caption{Reconstruction of simple phantom function using Landweber method and CGLS. No noise. Artefacts are present in Landweber iteration.}
\label{fig5}
\end{figure}
\begin{figure}[!h]
\begin{subfigure}{0.47\textwidth}
\includegraphics[width=1.0\linewidth, height=0.9\linewidth]{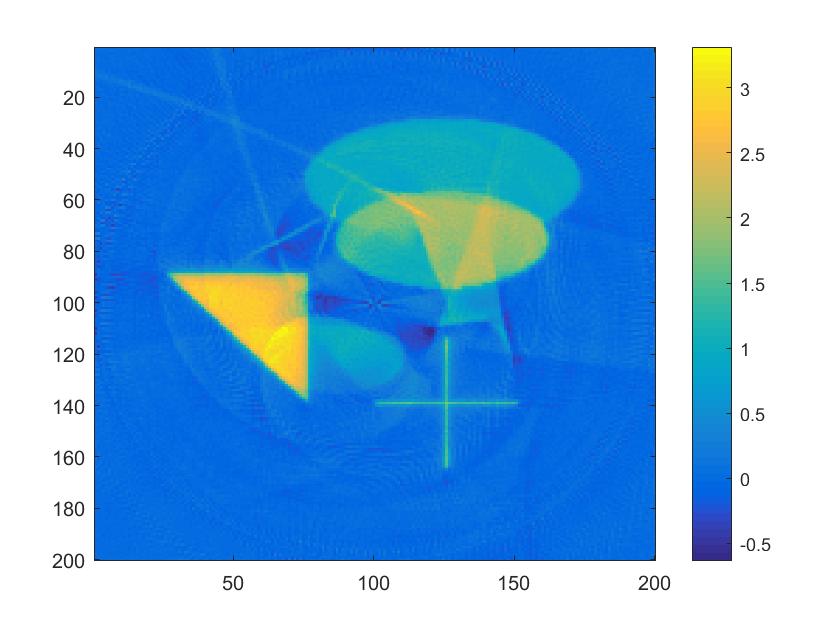} 
\end{subfigure}
\begin{subfigure}{0.47\textwidth}
\includegraphics[width=1.0\linewidth, height=0.9\linewidth]{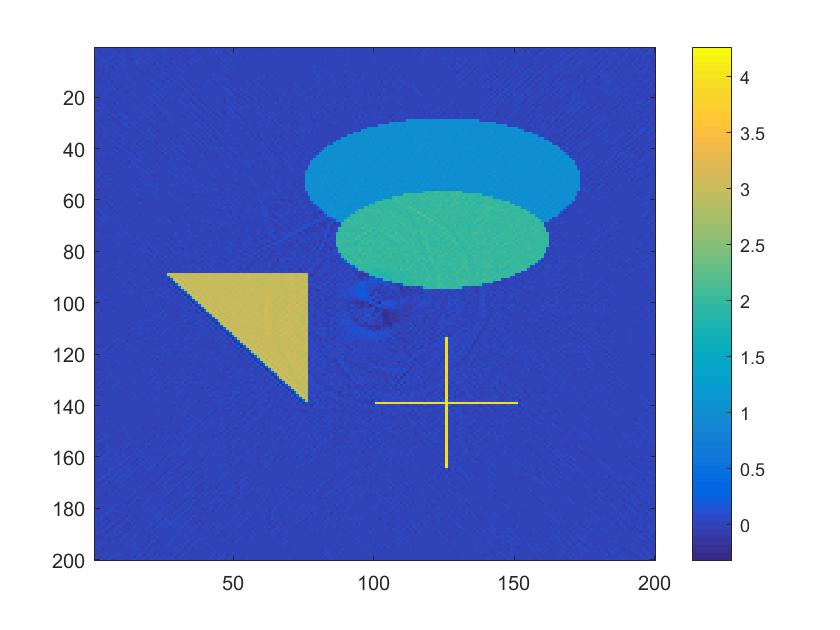}
\end{subfigure}
\caption{Reconstruction of complex phantom function using Landweber method and CGLS. No noise. Artefacts are present in Landweber iteration.}
\label{fig5.1}
\end{figure}


\clearpage
\appendix
\section{Potential application in airport baggage screening}
\label{app1}
Here we explain in more detail the proposed
application in airport baggage screening, and how the theory and
reconstruction methods presented in the main text relate to this field. In figure \ref{figRTT} we have displayed a machine configuration for RTT X-ray scanning in airport security screening (such a design is in use at airports today). The density $f$ is translated in the $x_3$ direction (out of the page) on a conveyor belt, and illuminated by a ring (the blue circle) of fixed-switched monochromatic (energy $E$) fan beam X-ray sources. The scattered intensity is then collected by a second ring (the green circle) of fixed energy-resolved detectors. The source and detector rings are coloured as in figures \ref{fig0} and \ref{fig1}.
\begin{figure}[!h]
\centering
\begin{tikzpicture}[scale=4]
\draw [thick, green] (0,0) circle [radius=0.7];
\draw [red, very thin] (-0.3,-0.2) rectangle (0.3,0.2);
\draw [very thin, dashed] (0,-1)--(0,0.7);
\draw [very thin, dashed] (0,-1)--(0.1,0.6928);
\draw [very thin, dashed] (0,-1)--(-0.1,0.6928);
\draw [very thin, dashed] (0,-1)--(0.2,0.6708);
\draw [very thin, dashed] (0,-1)--(-0.2,0.6708);
\draw [very thin, dashed] (0,-1)--(0.3,0.6324);
\draw [very thin, dashed] (0,-1)--(-0.3,0.6324);
\draw [very thin, dashed] (0,-1)--(0.4,0.5744);
\draw [very thin, dashed] (0,-1)--(-0.4,0.5744);
\draw [very thin] (0.9,0.9)--(0.3,0.6324);
\node at (1.1,1) {detector ring};
\draw [thick, blue] (0,0) circle [radius=1];
\draw [very thin] (-1.1,-0.9)--(-0.7,-0.71);
\node at (-1.1,-1) {source ring};
\draw [thick] (-0.4,-0.3)--(0.4,-0.3);
\draw [very thin] (-0.9,0.9)--(-0.3,0);
\node at (-0.9,1) {scanned object ($f$)};
\draw [very thin] (1.1,-0.9)--(0.3,-0.3);
\node at (1.1,-1) {conveyor belt};
\draw [red] circle [radius=0.57];
\draw [very thin] (1.3,0.3)--(0.57,0);
\node at (1.4,0.4) {scanning tunnel};
\draw [->,line width=1pt] (0,0)--(1.2,0)node[right]{$x_1$};
\draw [->,line width=1pt] (0,0)--(0,1.2)node[left]{$x_2$};
\end{tikzpicture}
\caption{A security scanning machine configuration is displayed. The source-detector ring offset is small and is modelled as zero.}
\label{figRTT}
\end{figure}
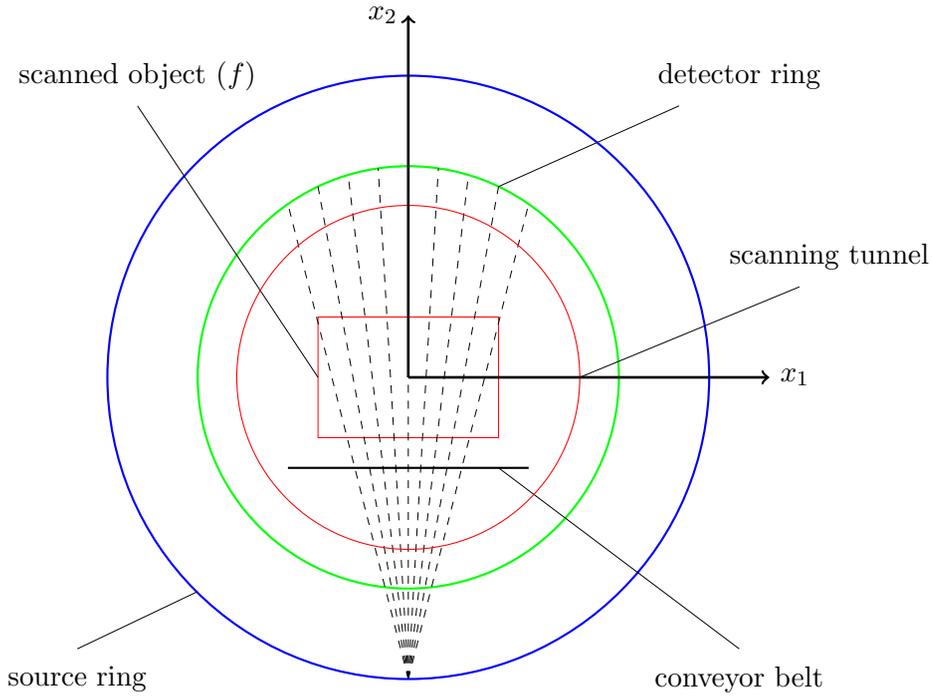

As is noted in the introduction (paragraph 3), the data are three dimensional. That is we can vary a source and detector position $(\vs,\vd)\in S^1\times S^1$ and the scattered energy $E'$ (since the detectors are energy-resolved). We consider the two dimensional subset of this data, when $\vs=-\vd$. Varying the source position $\vs$ (or $\vd$) corresponds to varying $\theta$ as in section \ref{tsec}. The scattered energy $E'$ determines $\cos\omega$ by equation \eqref{equ1} and in turn determines the torus radius
$$r=\frac{2}{\sqrt{1-\cos^2\omega}}.$$

The machine design of figure \ref{figRTT} has the
ability to measure a combination of transmission (straight through
photons) and scattered data. The photon counts measured when $E'=E$ (unattenuated photons) correspond to line integrals over the
attenuation coefficient $\mu_E$ (such as in standard transmission X-ray CT). The
Compton scattered data (for $E'<E$) determines the electron density
$f=n_e$ (by the theory of section \ref{tsec2}), and thus provides
additional information regarding the physical properties of the
scanned baggage. Hence we expect the use of the (extra) Compton data,
in conjunction with the transmission data, to allow for a more
accurate materials characterization (when compared to transmission or Compton tomography
separately) and to ultimately lead to a
more effective threat detection algorithm (e.g. reducing false
positive rates in airport screening). Such ideas have already been put
forward in \cite{me}, where a combination of $\mu_E$ and $n_e$
information is used to determine the effective atomic number of the
material.

\section{Additional reconstructions with analytic data}
\label{app2}
Here we present additional reconstructions with analytically generated $\Tc f$ data, using the same reconstruction method as before, minimizing the functional \eqref{recmth}. We consider the multiple ring phantom
\begin{equation}
f(\vx)=\sum_{j=1}^6j\chi_{B_{10,15}}\left(\vx-50\left(\cos\frac{j\pi}{3},\sin\frac{j\pi}{3}\right)\right)
\end{equation}
as displayed in figure \ref{A1}. Here $\chi_S$ denotes the characteristic function on $S$ and the reconstruction space is $[-100,100]^2$.
\begin{figure}[!h]
\begin{subfigure}{0.32\textwidth}
\includegraphics[width=0.9\linewidth, height=4cm]{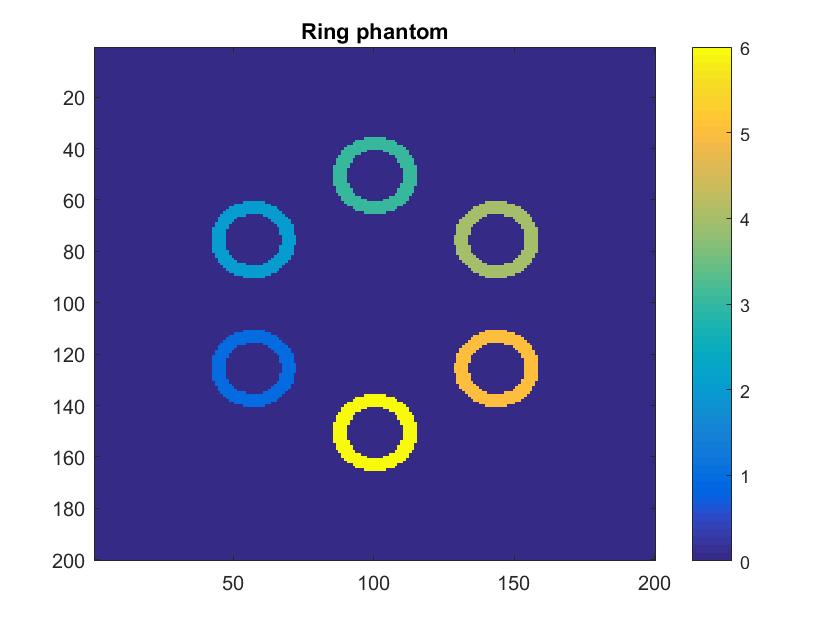}
\end{subfigure}
\begin{subfigure}{0.32\textwidth}
\includegraphics[width=0.9\linewidth, height=4cm]{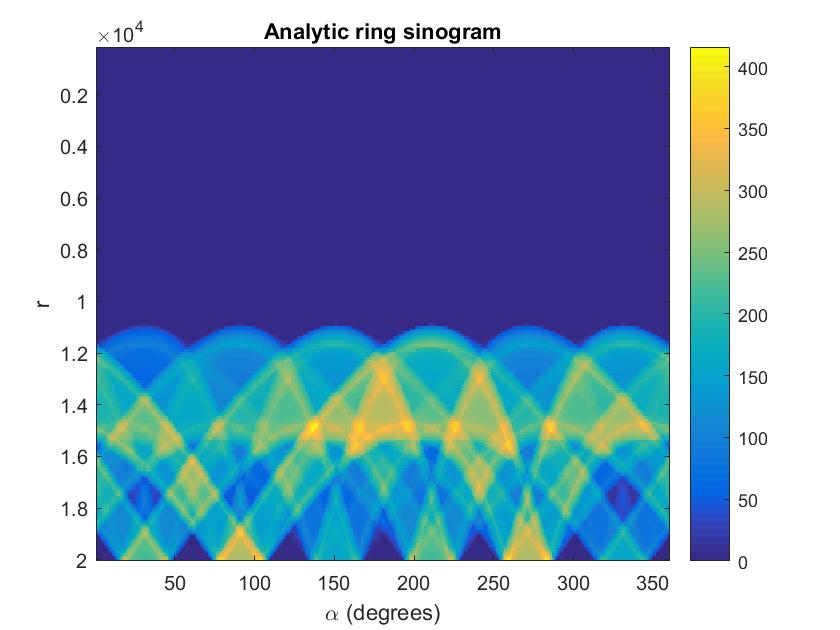} 
\end{subfigure}
\begin{subfigure}{0.32\textwidth}
\includegraphics[width=0.9\linewidth, height=4cm]{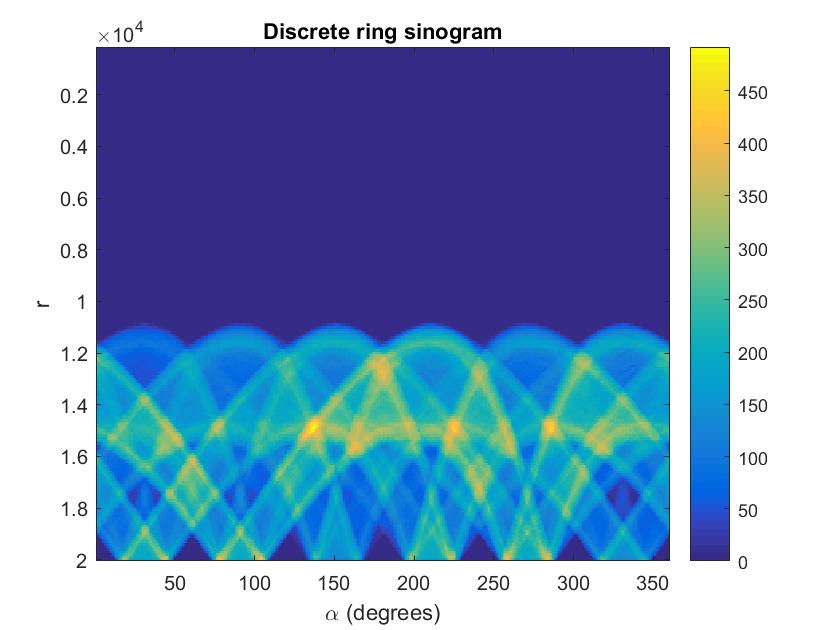}
\end{subfigure}
\caption{Ring phantom (left),  analytic sinogram (middle) and discrete sinogram (right)}
\label{A1}
\end{figure}
In this case the data are simulated as $\vb=\Tc
f(r,\alpha)$ for rotation angles $\alpha\in\left\{\frac{j\pi}{180} :
1\leq j\leq 360\right\}$ and for circle radii
$r\in\left\{\frac{j^2+200^2}{2j} : 1\leq j\leq 199\right\}$ (as in
section \ref{res}), and a Gaussian noise is added thereafter (as in
equation \eqref{noise}). See figure \ref{A1} for a comparison of the
analytic and discrete sinogram data. The discrete sinograms were
generated as before using $\vb=A\vv$ ($\vv$ is the discrete form of
$f$). The relative sinogram error is $\epsilon=\|\Tc f -
A\vv\|_2/\|\Tc f\|_2=0.11$, so in this case there is a significant
(systematic) error due to discretization. See figure \ref{A2} for
reconstructions of $f$ using the three methods considered in the main
text, namely Conjugate Gradient Least Squares (CGLS) with Tikhonov,
Landweber and heuristic Total Variation (TV). We present
reconstructions using analytic data with added noise and discrete data
with added noise for comparison. As in section \ref{res} we see the
best performance using heuristic TV. However there are additional
artefacts in the analytic reconstructions due to discretization
errors. Based on these experiments, it would be of benefit to
construct the discrete form of $\Tc$ ($A$) from exact circle-pixel
length intersections (as opposed to $A$ being a binary matrix).
However we leave this for further work.
\begin{figure}[!h]
\begin{subfigure}{0.32\textwidth}
\includegraphics[width=0.9\linewidth, height=4cm]{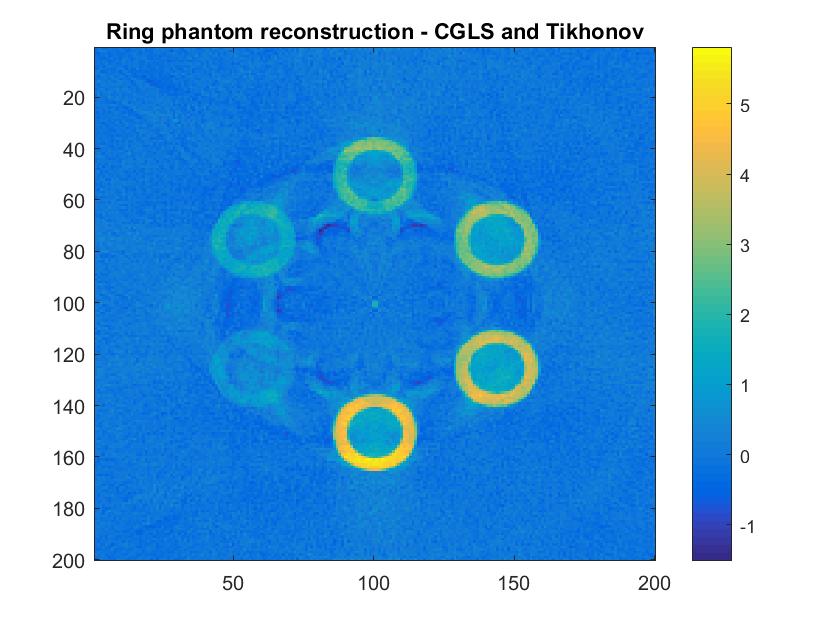}
\end{subfigure}
\begin{subfigure}{0.32\textwidth}
\includegraphics[width=0.9\linewidth, height=4cm]{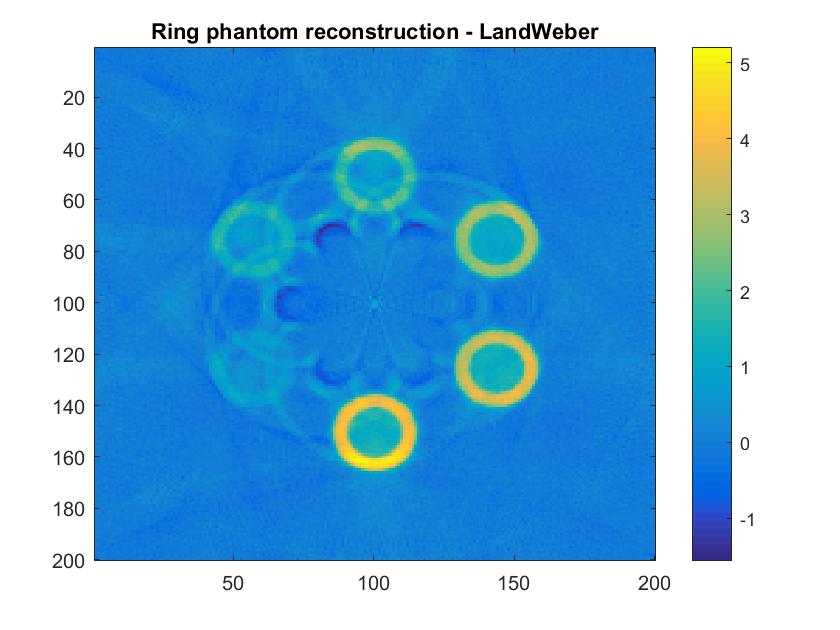} 
\end{subfigure}
\begin{subfigure}{0.32\textwidth}
\includegraphics[width=0.9\linewidth, height=4cm]{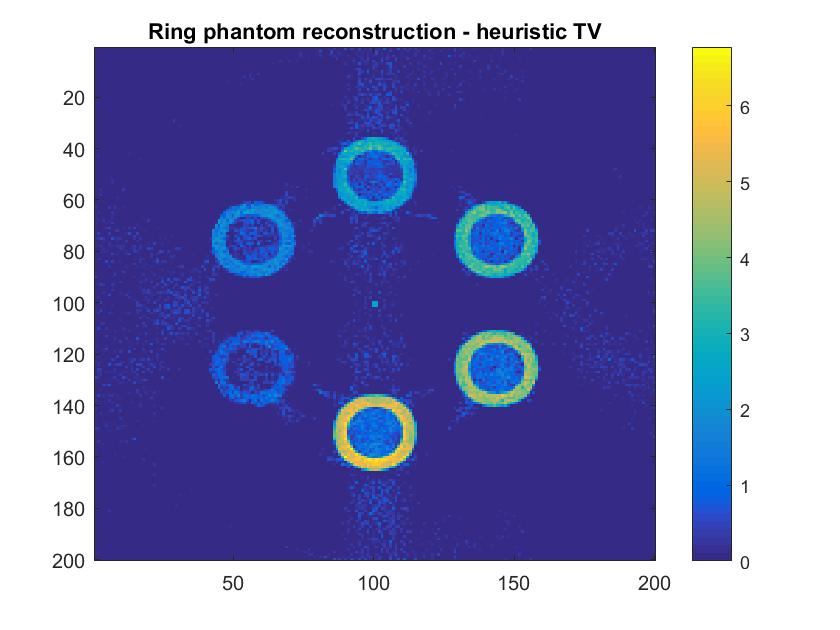}
\end{subfigure}

\begin{subfigure}{0.32\textwidth}
\includegraphics[width=0.9\linewidth, height=4cm]{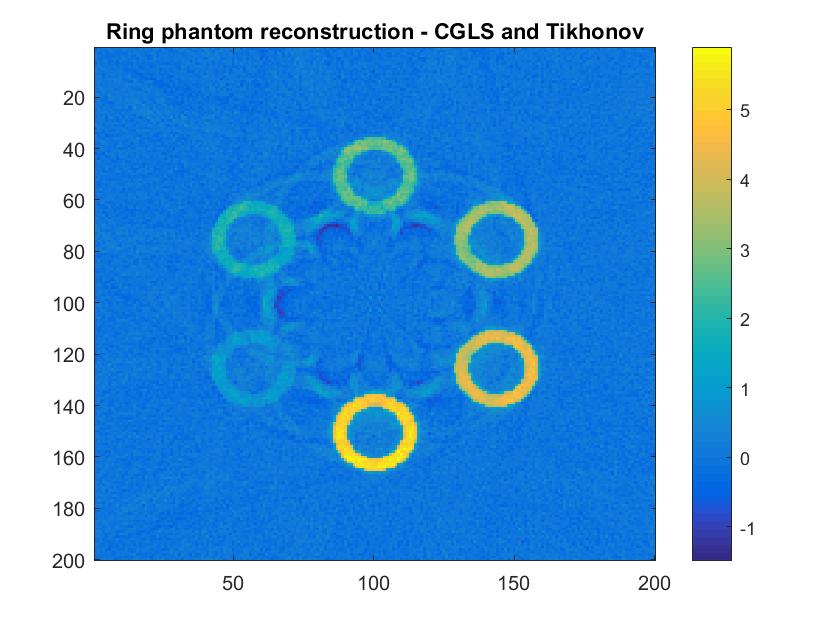}
\end{subfigure}
\begin{subfigure}{0.32\textwidth}
\includegraphics[width=0.9\linewidth, height=4cm]{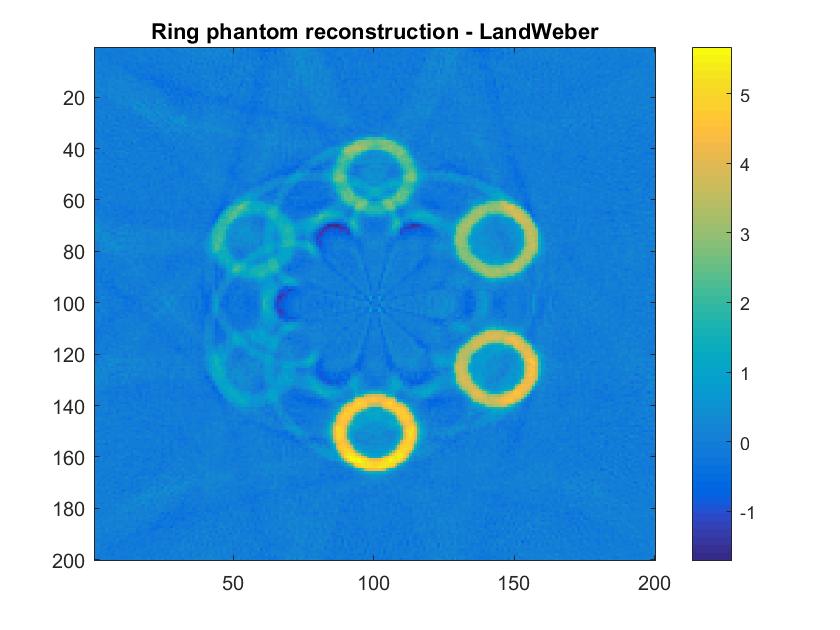} 
\end{subfigure}
\begin{subfigure}{0.32\textwidth}
\includegraphics[width=0.9\linewidth, height=4cm]{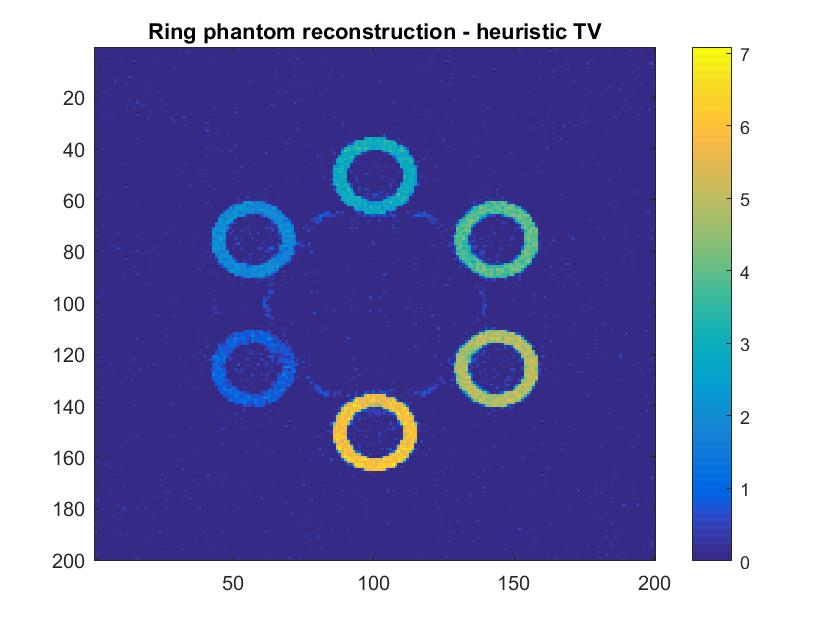}
\end{subfigure}
\caption{Top row --  Reconstructions with analytic data plus $5\%$ noise. Bottom row -- Reconstructions with discrete data plus $5\%$ noise (inverse crime).}
\label{A2}
\end{figure}

\section*{Acknowledgments}
 The authors thank Ga\"el Rigaud for
stimulating discussions about this research, in particular, about data
acquisition methods and the conversation that motivated Remark
\ref{rem:strength}.  The authors thank Eric Miller and his group for
providing a stimulating, supportive environment to do this research
and for providing the practical motivation for this work.  Finally, we
thank the journal editor for handling the article efficiently and the
referees for thoughtful, careful, insightful comments that improved
the article and helped clarify the proof of Theorem
\ref{thm:IntegralEqun}.  The work of the second author was partially
supported by U.S.\ National Science Foundation grant DMS 1712207. The first author was supported by the U.S. Department of Homeland Security, Science and Technology Directorate, Office of University Programs, under Grant Award 2013-ST-061-ED0001. The views and conclusions contained in this document are those of the authors and should not be interpreted as necessarily representing the official policies, either expressed or implied, of the U.S. Department of Homeland Security.

\bibliographystyle{siamplain}
\bibliography{references}

\end{document}